\documentclass [smallextended]{svjour3}[]

\usepackage{appendix}
\usepackage[utf8]{inputenc}
\usepackage[T1]{fontenc}
\usepackage{amsmath,amsfonts,amssymb}
\usepackage{algorithm}
\usepackage[noend]{algpseudocode}
\usepackage{graphicx}
\usepackage{caption}
\usepackage{subcaption}
\usepackage{diagbox}
\usepackage{url}
\usepackage{multicol}
\usepackage{xspace}
\usepackage{pdfpages}
\usepackage{bm}
\usepackage{mathabx}

\usepackage[citestyle=numeric,natbib=true,backend=bibtex]{biblatex}

\usepackage{tikz}
\usetikzlibrary{shapes}
\usetikzlibrary{trees}
\usetikzlibrary{arrows}
\usetikzlibrary{patterns}

\newcommand{\milp}{mixed integer linear program\xspace}
\newcommand{\milps}{mixed integer linear programs\xspace}

\algrenewcommand\algorithmicrequire{\textbf{Input:}}
\algrenewcommand\algorithmicensure{\textbf{Output:}}

\begin{document}


\title{On the integration of Dantzig-Wolfe and Fenchel decompositions via directional normalizations}

\author{François Lamothe \and Alain Haït \and Emmanuel Rachelson \and Claudio Contardo \and Bernard Gendron}

\institute{François Lamothe \and Alain Haït \and Emmanuel Rachelson
\at ISAE-SUPAERO, Université de Toulouse \\\email{francois.lamothe@isae.fr}
\and
Claudio Contardo
\at Concordia University, Canada
\and
Bernard Gendron
\at Université de Montréal, Canada
}

\maketitle
        
\begin{abstract}
The strengthening of linear relaxations and bounds of mixed integer linear programs has been an active research topic for decades. Enumeration-based methods for integer programming like linear programming-based branch-and-bound exploit strong dual bounds to fathom unpromising regions of the feasible space. In this paper, we consider the strengthening of linear programs via a composite of Dantzig-Wolfe and Fenchel decompositions. We provide geometric interpretations of these two classical methods. Motivated by these geometric interpretations, we introduce a novel approach for solving Fenchel sub-problems and introduce a novel decomposition combining Dantzig-Wolfe and Fenchel decompositions in an original manner. We carry out an extensive computational campaign assessing the performance of the novel decomposition on the unsplittable flow problem. Very promising results are obtained when the new approach is compared to classical decomposition methods.

\end{abstract}

\section*{Acknowledgments}
This document is the result of a research project funded by the \textit{Centre national d'études spatiales} (CNES) and Thales Alenia Space. C. Contardo and B. Gendron thank the Natural Sciences and Engineering Research Council of Canada (NSERC) for its financial support under Discovery Grants no. 2020-06311 and 2017-06054.

\section{Introduction}

Enumeration-based algorithms are arguably the main algorithmic frameworks to solve mixed-integer linear programs among which the branch-and-bound (B\&B) method \citep{land1960automatic} is perhaps the most efficient and versatile one. It relies on the ability to compute primal and dual bounds on the value of the solutions to the problem at hand. The method is most efficient when quick computing of tight primal and dual bounds is available, resulting in short enumerations. Traditionally, primal bounds are found via heuristics and dual bounds via solving relaxations of the problem. The most classical such relaxation is the so-called \textit{linear relaxation} in which the integrality constraints are ignored. The resulting problem, an ordinary linear program, can be solved efficiently either in polynomial time via interior point methods \citep{KHACHIYAN198053, mehrotra1992implementation}, or via a greedy method --- namely the simplex method \citep{dantzig1951maximization} --- of exponential worst-case complexity, but very efficient in practice \citep{borgwardt2012simplex}.

A common practice to improve the linear relaxation of mixed integer linear programs is to apply a decomposition method to a subset of the problem's constraints whose associated polyhedron does not have the integrality property (some vertices of the polyhedron have non-integer components). The decomposition tightens the polyhedron which in turn strengthens the linear relaxation of the problem. One of the most classical decomposition methods is the one by \citet{dantzig1960decomposition}, known as the Dantzig-Wolfe decomposition. It has proven successful in various applications which explains the attention it has received over the years \citep{desaulniers2006column}. The Dantzig-Wolfe decomposition is closely related to the Lagrangian decomposition which has also been successful in many practical applications. The main difference between the Lagrangian and Dantzig-Wolfe decomposition is that in the former only dual information is extracted and exploited in a sub-gradient algorithm whereas the latter also extracts primal information which can be embedded within an enumeration scheme, usually referred to as \textit{branch-and-price} \citep{barnhart1998branch}. Because in this work we will consider column generation and cutting plane procedures, we will focus on the Dantzig-Wolfe decomposition.

Our analysis restricts to decompositions that are able to exploit integer subproblems ---as opposed to pure linear problems---, and for that reason, we do not consider Benders decomposition \citep{benders1962partitioning}. Fenchel decomposition, on the other hand, is a cutting plane method similar to Benders decomposition that can be applied when the sub-problem is integer. It is thus able to improve the relaxation quality of \milps. This technique has had success in several applications such as knapsack problems \citep{boyd1993generating, kaparis2010separation}, generalized assignment problems \citep{avella2010computational}, network design problems with unsplittable flow \citep{chen2021exact} or even stochastic optimization problems \citep{ntaimo2013fenchel, beier2015stage}. 

Although Fenchel and Dantzig-Wolfe decompositions have been largely studied in the context of linear relaxation strengthening, several limitations of both decomposition methods have been identified in the literature which may affect their convergence. In particular, the Dantzig-Wolfe decomposition is known to suffer from the degeneracy of its master problem (which happens when the master problem admits several optimal dual solutions). This explains why a considerable effort has been made by the scientific community to improve these decomposition methods and overcome their weaknesses. This work is in this line of contribution. We show that Dantzig-Wolfe and Fenchel decompositions can be seen in a similar light as constructing inner and outer approximations of the polyhedron being decomposed and that the synergy between the two approximations may be beneficial for the entire approach when applied to some large-scale mixed-integer linear programs.

We outline the main contributions of our work as follows: 
\begin{itemize}
    \item We provide geometric interpretations of Dantzig-Wolfe and Fenchel decompositions which may fuel complementary insights on these two approaches compared to purely analytical interpretations.
    \item We provide a critical overview of normalizations at the core of the separation problems arising in Fenchel decomposition. The normalization impacts the type of cut, its properties, and ultimately the convergence speed of the decomposition method. We provide geometric interpretations of several types of normalization and their properties.
    \item We introduce a novel approach to the Fenchel sub-problem when a directional normalization is used. The proposed method possesses reduces the numerical instabilities of a direct resolution approach commonly used. We show that the new approach solves the Fenchel sub-problem in finitely many iterations.
    \item We introduce a new decomposition method inspired by both the Dantzig-Wolfe and the Fenchel decompositions. The proposed method uses a Dantzig-Wolfe master problem and a Fenchel master problem. A Fenchel sub-problem guided with a directional normalization is used to coordinate the two master problems. The resulting method is shown to perform especially well on instances presenting high degrees of degeneracy. We provide a possible explanation of this phenomenon based on our findings.
\end{itemize}

In the following, we start by giving in Section \ref{sec:context} the context in which we apply the various decomposition methods together with some notation. Then, we present in Section \ref{geometric_interpretation} a geometric interpretation of Dantzig-Wolfe and Fenchel decomposition methods. Section \ref{Fenchel_sp_normalization} is dedicated to an overview of the resolution of the Fenchel subproblem and in particular, the impact of the normalization used. This concept will be used in the next section for the new methods introduced in the paper. We then present in Section \ref{iterative_sub_problem} a new approach to solving the Fenchel subproblem. In Section \ref{hybrid_DW_Fenchel}, we introduce the new decomposition method which integrates both Dantzig-Wolfe and Fenchel master problems as well as a Fenchel subproblem. In Section \ref{sec:application}, we illustrate our method by applying it to the unsplittable flow problem. An experimental evaluation of all proposed methods is then made in Section \ref{sec:results} on small and medium-sized instances of the problem.

\section{Context}
\label{sec:context}

In the remainder of this paper, we will apply decomposition methods to the following general \milp:
\begin{subequations}
\begin{alignat}{3}
(P) \quad &\max_{x} && c x  \\
&\text{subject to} \quad && A_1 x \leq b_1 \\
& && A_2 x \leq b_2 \\
& && x \in X
\end{alignat}
\end{subequations}
where $ X $ is a product set $ \mathbb{R}^p \times \mathbb{Z}^q $ of appropriate dimension whose linear relaxation $\mathbb{R}^{p+q}$ will be noted $ \bar{X} $. Moreover, to explain the methods used, we will consider the following polyhedra which are assumed to be bounded to simplify the explanations:
\begin{subequations}
\begin{alignat*}{2}
& LR_1 = \{x \in \bar{X} | A_1 x \leq b_1 \}  \\
& LR_2 = \{x \in \bar{X} | A_2 x \leq b_2 \}  \\
& Q_2 = conv(\{x \in X | A_2 x \leq b_2 \})
\end{alignat*}
\end{subequations}
In order to bound from above the value of the optimal solution of a general \milp as ($P$) one can optimize its objective function over a relaxation containing the set of valid solutions. Usually, the linear relaxation of the set of solutions $ LR_1 \cap LR_2 $ is used. However, one might want a tighter relaxation to obtain a better upper bound. This can be obtained by using the relaxation $LR_1 \cap Q_2$ which still contains the set of valid solutions while being included in $ LR_1 \cap LR_2 $. As illustrated in Figure \ref{polyedron}, this relaxation usually returns strictly better bounds than the linear relaxation when the polyhedron $Q_2$ is strictly included in $LR_2$ which happens when the polyhedron $LR_2$ does not have the integrality property (some of the vertices of the polyhedron have non-integer coordinates). However, the drawback of $Q_2$ compared to $LR_2$ is that one usually only has a representation of $Q_2$ with an exponential number of variables or constraints which is not manageable directly by a linear programming solver. In compensation, we assume to have at our disposal an efficient algorithm to optimize a linear function on the polyhedron $Q_2$. This algorithm is called the optimization oracle and solves the following \milp:
\begin{subequations}
\begin{alignat*}{3}
(O) \quad &\max_{x}      &&\pi x \\
&\text{subject to} \quad && A_2 x \leq b_2 \\
&                        && x \in X
\end{alignat*}
\end{subequations}
where $ x \mapsto \pi x $ is any linear function. Thus, to be able to optimize over $ LR_1 \cap Q_2 $, the goal of a decomposition method is to use this oracle to compute an approximation of $Q_2$ of manageable size. In particular, the Dantzig-Wolfe decomposition iteratively grows an inner approximation of $Q_2$ while the Fenchel decomposition iteratively refines an outer approximation of $Q_2$.

\begin{figure}[ht]
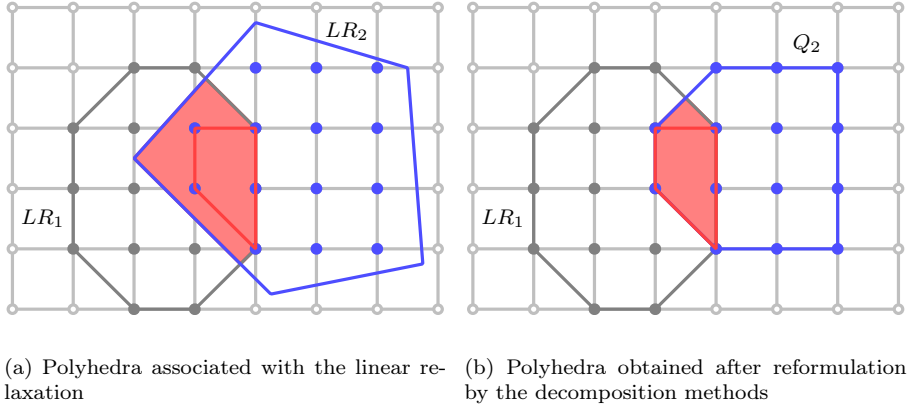

     \centering
     \begin{subfigure}[b]{0.49 \columnwidth}
         \centering
         \begin{tikzpicture}[scale=0.8]
\input{figure_tikz/tikz_settings}
\input{figure_tikz/Components/grid_no_arrow}
\input{figure_tikz/Components/fill_intersection}
\input{figure_tikz/Components/poly_flots}
\input{figure_tikz/Components/poly_capa}
\input{figure_tikz/Components/poly_integer_sol}

\node at (1.5,2.5) {\small $LR_1$};
\node at (6.5,5.6) {\small $LR_2$};
\end{tikzpicture}
         \caption{Polyhedra associated with the linear relaxation}
     \end{subfigure}
     \hfill
     \begin{subfigure}[b]{0.49 \columnwidth}
         \centering
         \begin{tikzpicture}[scale=0.8]
\input{figure_tikz/tikz_settings}
\input{figure_tikz/Components/grid_no_arrow}
\input{figure_tikz/Components/fill_approx}
\input{figure_tikz/Components/poly_flots}
\input{figure_tikz/Components/poly_approx}
\input{figure_tikz/Components/poly_integer_sol}

\node at (1.5,2.5) {\small $LR_1$};
\node at (6.5,5.4) {\small $Q_2$};
\end{tikzpicture}
         \caption{Polyhedra obtained after reformulation by the decomposition methods}
     \end{subfigure}
    \caption{Representation of solution spaces of the linear relaxation and the models obtained from the decomposition methods}
    \label{polyedron}
\end{figure}

Decomposition methods are known to work very well when the matrix $A_2$ is block diagonal because it enables decomposing the problem of optimizing over $Q_2$ into several smaller problems (one per block of  the matrix).  For instance, in the case of the application considered in this study, the unsplittable flow problem, the capacity constraints lead to a block diagonal structure. In this context, the constraints $A_1 x \leq b_1$ would correspond to the flow conservation constraints while the constraints $A_2 x \leq b_2$ would correspond to the capacity constraints. However, our work makes use of a different rationale. Let us assume that the matrix $A_2$ is sparse with non-zero elements for only a few variables; this may happen for instance when $A_2$ corresponds to one of the blocks of a block diagonal matrix. With this choice, it is possible to apply decomposition methods even in contexts where the main problem ($P$) does not have any block diagonal structure. For clarity purposes, we will present the methods as if we were decomposing only one polyhedron at a time. However, in practice, one would decompose several polyhedra at the same time; for example, all the blocks of a block diagonal matrix.

\section{Geometric interpretation of Dantzig-Wolfe and Fenchel decompositions}
\label{geometric_interpretation}

In this section, we provide a geometric interpretation of both decomposition methods with the potential to fuel new intuitions regarding their strengths and weaknesses.

\subsection{Dantzig-Wolfe decomposition}
\label{sec:presentation-DW}
Instead of optimizing the objective function of \milp (P) over the solution set $LR_1 \cap LR_2$ of its linear relaxation, Dantzig-Wolfe decomposition allows for the optimization over the smaller set $LR_1 \cap Q_2$. Because one does not usually have a manageable description of the set $Q_2$, it is necessary to compute an approximation of $Q_2$ in order to optimize over the intersection $LR_1 \cap Q_2$. To compute this approximation, we assume the availability of an optimization oracle over $Q_2$. The main idea in the Dantzig-Wolfe decomposition is to iteratively grow an inner approximation of $Q_2$. One can create such an approximation by obtaining a set of points $x_i$ belonging to $Q_2$ (usually extreme points of $Q_2$) and setting the approximation as the convex envelope of these points. We will denote this inner approximation with $\widehat{Q}_2$. The Dantzig-Wolfe decomposition illustrated in Figure \ref{DW_illustration} proceeds as follows:
\begin{enumerate}
    \item Initialize the approximation $\widehat{Q}_2$ with points of $Q_2$
    \item Find the optimal solution $\widehat{x}$ over $LR_1 \cap \widehat{Q}_2$
    \item Search for a point of $Q_2$ whose addition to $\widehat{Q}_2$ may improve the value of the optimal solution $\widehat{x}$
    \item If such a point exist, add it to $\widehat{Q}_2$ and go to Step 2.
    \item Else: $\widehat{x}$ is the optimal solution over $LR_1 \cap Q_2$. Stop.
\end{enumerate}

\subsubsection{Optimizing over $LR_1 \cap \widehat{Q}_{2}$} To optimize over $LR_1 \cap \widehat{Q}_2$, a master linear program is created in which the condition $x \in \widehat{Q}_2$ must be enforced. This can be done by rewriting $x$ as a convex combination of its extreme points $x_i$. To that end, a variable $ \lambda_i $ is introduced for each point $ x_i $ used to create $\widehat{Q}_2$. This variable represents the weight of the vertex $ x_i $ in the convex combination. The Dantzig-Wolfe master linear program can then be written as follows:
\begin{subequations}
\begin{alignat*}{3}
(DW) \quad &\max_{x, \lambda_i} && c^{T}x  \\
&\text{subject to} \quad && A_1 x \leq b_1 \\
&&& x = \sum_{i \in I} \lambda_i x_i\\
&&& \sum_{i \in I} \lambda_i = 1\\
&&& x \in \bar{X}, ~ \lambda_i \in \mathbb{R}^+, \quad \forall i \in I
\end{alignat*}
\end{subequations}

\begin{figure}[!ht]
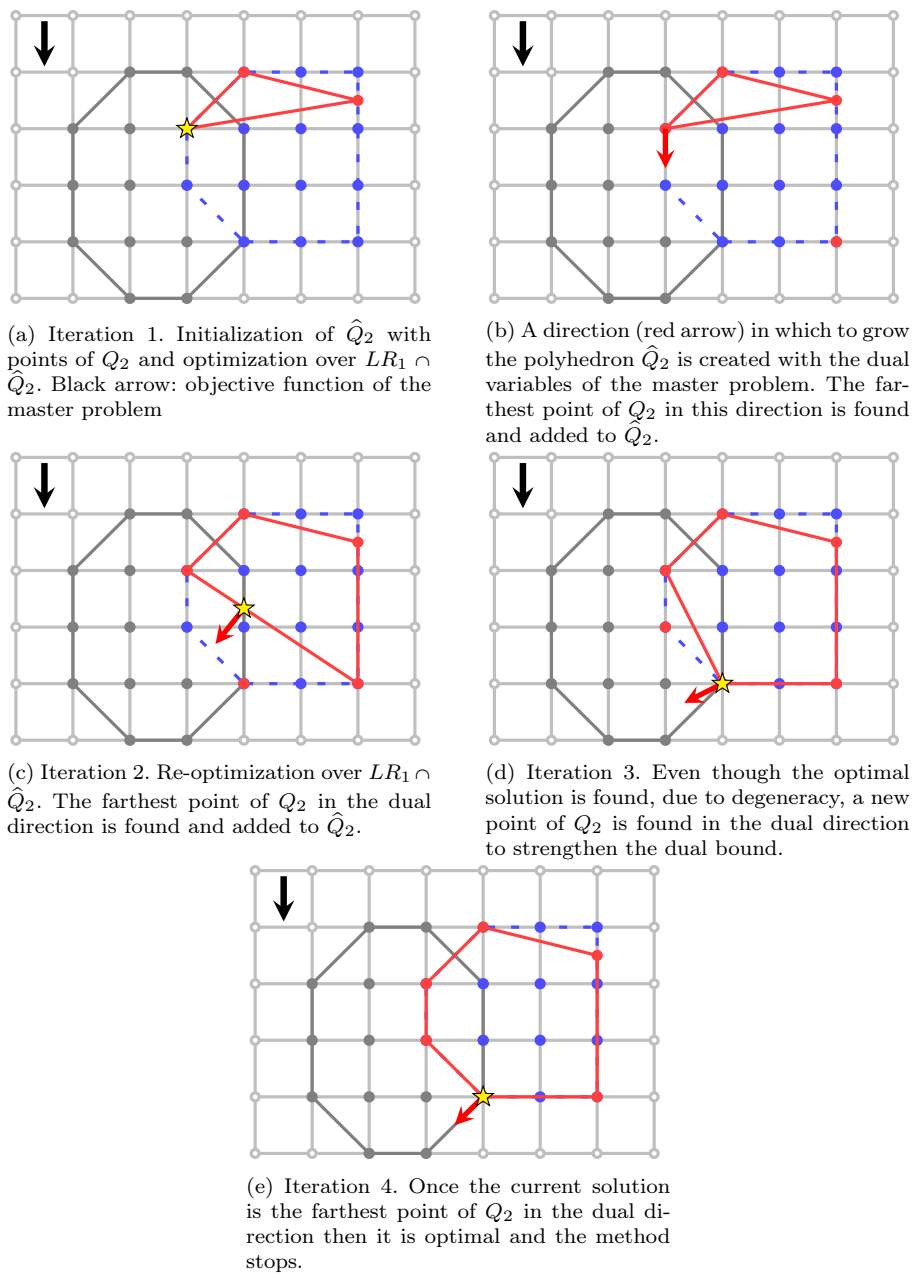

     \centering
     \begin{subfigure}[t]{0.47 \columnwidth}
         \centering
         \begin{tikzpicture}[scale=0.75]
\input{figure_tikz/tikz_settings}
\input{figure_tikz/Components/grid}
\input{figure_tikz/Components/poly_flots}
\input{figure_tikz/Components/poly_approx_dashed}
\input{figure_tikz/Components/poly_DW_1}

\node[star, star points=5, draw, star point ratio=2.5, scale=0.4, fill=yellow] () at (4,4){};
\end{tikzpicture}
         \caption{Iteration 1. Initialization of $\widehat{Q}_2$ with points of $Q_2$ and optimization over $LR_1 \cap \widehat{Q}_2$. Black arrow: objective function of the master problem}
     \end{subfigure}
     \hfill
     \begin{subfigure}[t]{0.47 \columnwidth}
         \centering
         \begin{tikzpicture}[scale=0.75]
\input{figure_tikz/tikz_settings}
\input{figure_tikz/Components/grid}
\input{figure_tikz/Components/poly_flots}
\input{figure_tikz/Components/poly_approx_dashed}
\input{figure_tikz/Components/poly_DW_1}

\draw[-stealth, red, line width=0.7mm] (4, 4) -- (4, 3.3);
\draw[my-red, fill=my-red, very thick] (7,2) circle(0.07cm);
\end{tikzpicture}
         \caption{A direction (red arrow) in which to grow the polyhedron $ \widehat{Q}_2$ is created with the dual variables of the master problem. The farthest point of $Q_2$ in this direction is found and added to $ \widehat{Q}_2$.}
     \end{subfigure}
     \begin{subfigure}[t]{0.47 \columnwidth}
         \centering
         \begin{tikzpicture}[scale=0.75]
\input{figure_tikz/tikz_settings}
\input{figure_tikz/Components/grid}
\input{figure_tikz/Components/poly_flots}
\input{figure_tikz/Components/poly_approx_dashed}
\input{figure_tikz/Components/poly_DW_2}

\draw[-stealth, red, line width=0.7mm] (5, 3.33) -- (4.5, 2.7);
\node[star, star points=5, draw, star point ratio=2.5, scale=0.4, fill=yellow] () at (5,3.33){};
\draw[my-red, fill=my-red, very thick] (5,2) circle(0.07cm);
\end{tikzpicture}
         \caption{Iteration 2. Re-optimization over $LR_1 \cap \widehat{Q}_2$. The farthest point of $Q_2$ in the dual direction is found and added to $ \widehat{Q}_2$.}
     \end{subfigure}
     \hfill
     \begin{subfigure}[t]{0.47 \columnwidth}
         \centering
         \begin{tikzpicture}[scale=0.75]
\input{figure_tikz/tikz_settings}
\input{figure_tikz/Components/grid}
\input{figure_tikz/Components/poly_flots}
\input{figure_tikz/Components/poly_approx_dashed}
\input{figure_tikz/Components/poly_DW_3}

\draw[-stealth, red, line width=0.7mm] (5,2) -- (4.33, 1.66);
\node[star, star points=5, draw, star point ratio=2.5, scale=0.4, fill=yellow] () at (5,2){};
\draw[my-red, fill=my-red, very thick] (4,3) circle(0.07cm);
\end{tikzpicture}
         \caption{Iteration 3. Even though the optimal solution is found, due to degeneracy, a new point of $Q_2$ is found in the dual direction to strengthen the dual bound.}
     \end{subfigure}
     \begin{subfigure}[t]{0.47 \columnwidth}
         \centering
         \begin{tikzpicture}[scale=0.75]
\input{figure_tikz/tikz_settings}
\input{figure_tikz/Components/grid}
\input{figure_tikz/Components/poly_flots}
\input{figure_tikz/Components/poly_approx_dashed}
\input{figure_tikz/Components/poly_DW_4}

\draw[-stealth, red, line width=0.7mm] (5,2) -- (4.5, 1.5);
\node[star, star points=5, draw, star point ratio=2.5, scale=0.4, fill=yellow] () at (5,2){};
\end{tikzpicture}
         \caption{Iteration 4. Once the current solution is the farthest point of $Q_2$ in the dual direction then it is optimal and the method stops.}
     \end{subfigure}
    \caption{Geometric illustration of the Dantzig-Wolfe decomposition}
    \label{DW_illustration}
\end{figure}

\subsubsection{Finding an improving point of $Q_2$} 

In the Dantzig-Wolfe decomposition, the master problem returns the farthest point $x^*$ of $LR_1 \cap \widehat{Q}_2$ in the direction $c$ and one would like to know whether this point is also the optimal for $LR_1 \cap Q_2$ or if the inner approximation $\widehat{Q}_2$ needs to be improved. The dual point of view of this statement, on which is based the Dantzig-Wolfe subproblem, is that the bound $c x \leq c x^*$ is valid for $LR_1 \cap \widehat{Q}_2$ and we would like to know if it also holds true for $LR_1 \cap Q_2$.

\textbf{Ideas from linear programming duality theory:} The following concepts are illustrated in Figure \ref{DW_bounds}. Linear programming duality informs us that because $LR_1 \cap \widehat{Q}_2$ is the intersection of two polyhedra, the bound $c x \leq c x^*$ can always be decomposed as the sum of two inequalities, one valid for $LR_1$ and the other valid for $\widehat{Q}_2$. Furthermore, the dual variables of the master problem yield such a decomposition of the optimal bound for $LR_1 \cap \widehat{Q}_2$. Indeed, let us denote $u$, $\pi$ and $\pi_0$ the optimal dual variables of the constraints $A_1x \leq b_1$, $x = \sum_i \lambda_i x_i$ and $\sum_i \lambda_i = 1$, respectively. By construction of the dual of the master problem (whose derivation we let to the reader), $uA_1x \leq ub_1$ is valid for $LR_1$, $\pi x \leq \pi_0$ is valid for $\widehat{Q}_2$ and these two inequalities sum to the optimal bound $c x \leq c x^*$ over $LR_1 \cap \widehat{Q}_2$ (note: the equality $ub + pi_0 = c x^*$ follows thanks to the strong duality theorem of linear programming \citep{matouvsek2007understanding}).

\begin{figure}[ht!]
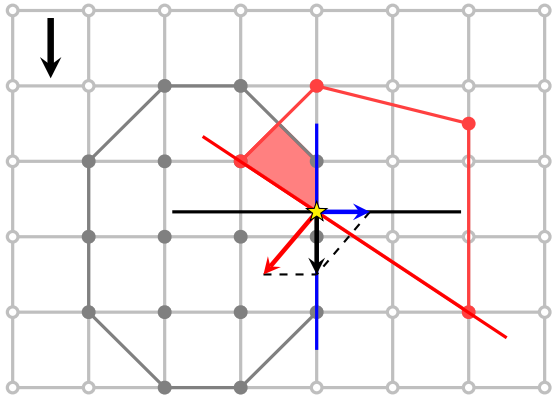

     \centering
     \begin{tikzpicture}[scale=1]
\input{figure_tikz/tikz_settings}
\input{figure_tikz/Components/grid}
\input{figure_tikz/Components/fill_inter_small}
\input{figure_tikz/Components/poly_flots}
\input{figure_tikz/Components/poly_DW_2}

\draw[black,very thick](3.1,3.33)--(6.9,3.33);
\draw[red,very thick](7.5,1.66)--(3.5,4.33);
\draw[blue,very thick](5,1.5)--(5,4.5);
\draw[black, dashed, thick](4.3, 2.5)--(5, 2.5);
\draw[black, dashed, thick](5.7, 3.33)--(5, 2.5);
\draw[ -stealth, red, line width=0.6mm] (5, 3.33) -- (4.3, 2.5);
\draw[-stealth, blue, line width=0.6mm] (5, 3.33) -- (5.7, 3.33);
\draw[-stealth, line width=0.6mm] (5, 3.33) -- (5, 2.5);
\node[star, star points=5, draw, star point ratio=2.5, scale=0.4, fill=yellow] () at (5,3.33){};
\end{tikzpicture}
    \caption{The optimal bound (black) on the optimization direction (black arrows) over the intersection (red fill) of $LR_1$ (grey) and $\widehat{Q}_2$ (red) can be decomposed as the sum of a valid inequality (blue) for $LR_1$ and a valid inequality (red) for $\widehat{Q}_2$.}
    \label{DW_bounds}
\end{figure}

\textbf{The Dantzig-Wolfe subproblem:} In order to show that the solution $x^*$ of the master problem is not the farthest point of $LR_1 \cap Q_2$ in the direction $c$, one must at least prove that the bound $c x \leq c x^*$ implied by the dual variables of the master problem is not valid for $LR_1 \cap Q_2$. However, if the inequality $\pi x \leq \pi_0$ were valid for $Q_2$, because we know that $uA_1x \leq ub_1$ is valid for $LR_1$ and that these two inequalities sum to the bound $c x \leq c x^*$ then this bound would be valid for $LR_1 \cap Q_2$. Thus, the goal of the sub-problem is to show that the inequality $\pi x \leq \pi_0$ is not valid for $Q_2$. To that end, the subproblem is tasked to find the farthest point of $Q_2$ in the direction $\pi$. If this point violates $\pi x \leq \pi_0$ then we have found a point of $Q_2$ violating the inequality. The point can then be added to the master problem to grow the inner approximation $\widehat{Q}_2$ and at least prevent the master problem from yielding the same dual variables again. Otherwise, the current solution is optimal because the bound implied by the dual variables is valid for $LR_1 \cap Q_2$.

\subsubsection{A note on degeneracy}
We have seen above that the dual variables of the master problem imply a bound on the value of its objective function. This bound can be considered a certificate that the current value of the master problem is optimal. Meanwhile, a linear program such as a Dantzig-Wolfe master problem is said to be degenerate when it has several dual optimal solutions. Each of these dual solutions is a certificate of optimality for the current value of the master problem. Thus, to improve the value of the master problem, one must invalidate each of these certificates. However, the subproblem of the Dantzig-Wolfe procedure only invalidates one of these certificates without any guarantee about the one implied by the other dual solutions. Thus the Dantzig-Wolfe decomposition is susceptible to having many iterations without improvement of the objective function when the master problem is highly degenerate. This can slow down the convergence of the method considerably.

\subsection{Fenchel decomposition}
\label{sec:presentation-Fenchel}

In Fenchel decomposition, instead of growing an inner approximation of $Q_2$, an outer approximation of the polyhedron $Q_2$ is refined to enable the optimization over $LR_1 \cap Q_2$. Such an outer approximation can use any collection of inequalities valid for $Q_2$. The decomposition is illustrated in Figure \ref{Fenchel_illustration} and proceeds as follows:
\begin{enumerate}
    \item Initialize an outer approximation $\widecheck{Q}_2$ with inequalities valid for $Q_2$; typically one can take $\widecheck{Q}_2 = LR_2$.
    \item Optimize over $LR_1 \cap \widecheck{Q}_2$ and recover a solution $\widecheck{x}$.
    \item Search for a cut separating $\widecheck{x}$ from $Q_2$.
    \item If such a cut exists, add the cut to $\widecheck{Q}_2$ and go to Step 2.
    \item Else: $\widecheck{x}$ is the optimal solution  over $LR_1 \cap Q_2$. Stop.
\end{enumerate}

\begin{figure}[ht!]
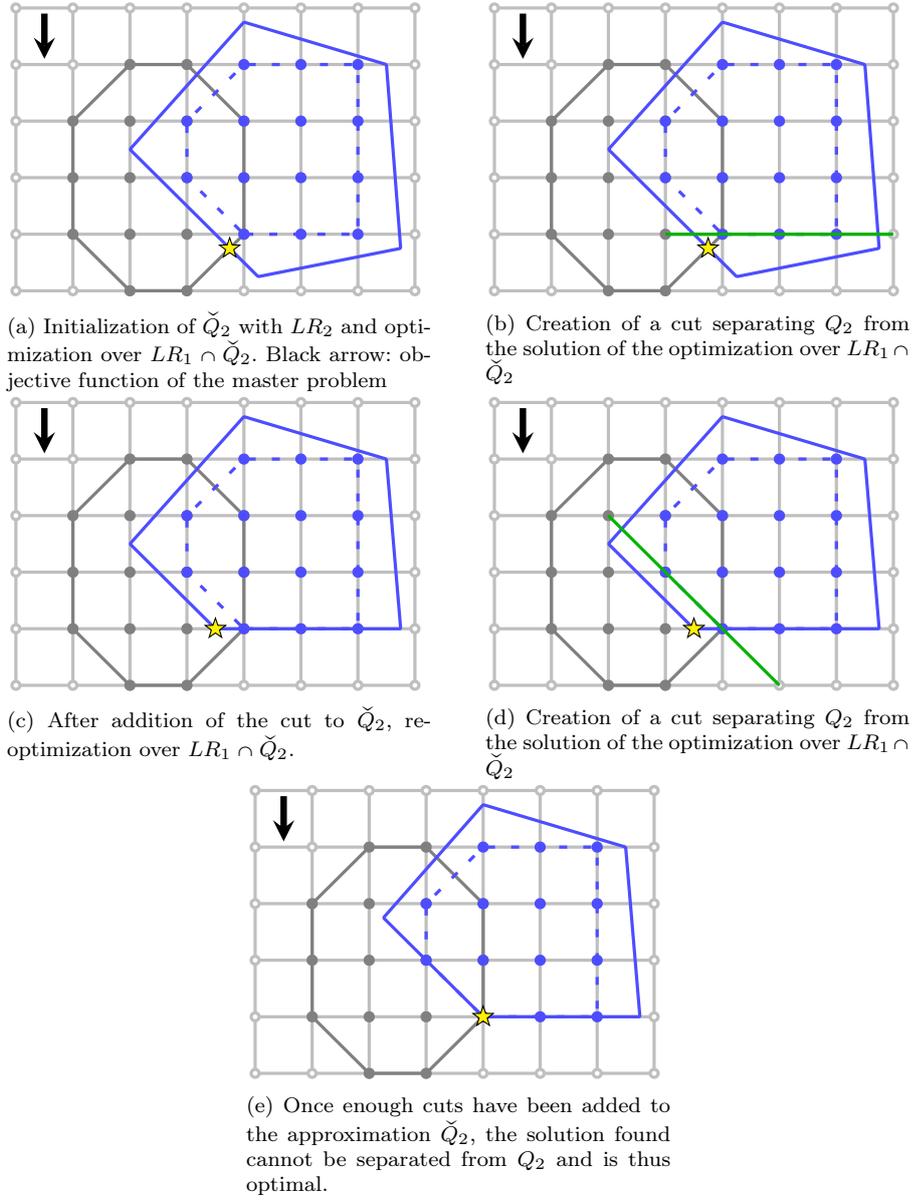

    \vspace{1cm}
     \centering
     \begin{subfigure}[t]{0.47 \columnwidth}
         \centering
         \begin{tikzpicture}[scale=0.75]
\input{figure_tikz/tikz_settings}
\input{figure_tikz/Components/grid}
\input{figure_tikz/Components/poly_flots}
\input{figure_tikz/Components/poly_approx_dashed}

\draw[dark-blue, very thick, rounded corners=0.5pt](5.25,1.25)--(7.75,1.75)--(7.5,5)--(5,5.75)--(3,3.5)--cycle;

\node[star, star points=5, draw, star point ratio=2.5, scale=0.4, fill=yellow] () at (4.75,1.75){};
\end{tikzpicture}
         \caption{Initialization of $\widecheck{Q}_2$ with $LR_2$ and optimization over $LR_1 \cap \widecheck{Q}_2$. Black arrow: objective function of the master problem}
     \end{subfigure}
     \hfill
     \begin{subfigure}[t]{0.47 \columnwidth}
         \centering
         \begin{tikzpicture}[scale=0.75]
\input{figure_tikz/tikz_settings}
\input{figure_tikz/Components/grid}
\input{figure_tikz/Components/poly_flots}
\input{figure_tikz/Components/poly_approx_dashed}

\draw[dark-blue, very thick, rounded corners=0.5pt](5.25,1.25)--(7.75,1.75)--(7.5,5)--(5,5.75)--(3,3.5)--cycle;

\draw[my-green, very thick, rounded corners=0.5pt](4,2)--(8,2);

\node[star, star points=5, draw, star point ratio=2.5, scale=0.4, fill=yellow] () at (4.75,1.75){};
\end{tikzpicture}
         \caption{Creation of a cut separating $Q_2$ from the solution of the optimization over $LR_1 \cap \widecheck{Q}_2$}
     \end{subfigure}
     \begin{subfigure}[t]{0.47 \columnwidth}
         \centering
         \begin{tikzpicture}[scale=0.75]
\input{figure_tikz/tikz_settings}
\input{figure_tikz/Components/grid}
\input{figure_tikz/Components/poly_flots}
\input{figure_tikz/Components/poly_approx_dashed}

\draw[dark-blue, very thick, rounded corners=0.5pt](4.5,2)--(7.75,2)--(7.5,5)--(5,5.75)--(3,3.5)--cycle;

\node[star, star points=5, draw, star point ratio=2.5, scale=0.4, fill=yellow] () at (4.5,2){};
\end{tikzpicture}
         \caption{After addition of the cut to $\widecheck{Q}_2$, re-optimization over $LR_1 \cap \widecheck{Q}_2$.}
     \end{subfigure}
     \hfill
     \begin{subfigure}[t]{0.47 \columnwidth}
         \centering
         \begin{tikzpicture}[scale=0.75]
\input{figure_tikz/tikz_settings}
\input{figure_tikz/Components/grid}
\input{figure_tikz/Components/poly_flots}
\input{figure_tikz/Components/poly_approx_dashed}

\draw[dark-blue, very thick, rounded corners=0.5pt](4.5,2)--(7.75,2)--(7.5,5)--(5,5.75)--(3,3.5)--cycle;

\draw[my-green, very thick, rounded corners=0.5pt](6,1)--(3,4);

\node[star, star points=5, draw, star point ratio=2.5, scale=0.4, fill=yellow] () at (4.5,2){};
\end{tikzpicture}
         \caption{Creation of a cut separating $Q_2$ from the solution of the optimization over $LR_1 \cap \widecheck{Q}_2$}
     \end{subfigure}
     \begin{subfigure}[t]{0.47 \columnwidth}
         \centering
         \begin{tikzpicture}[scale=0.75]
\input{figure_tikz/tikz_settings}
\input{figure_tikz/Components/grid}
\input{figure_tikz/Components/poly_flots}
\input{figure_tikz/Components/poly_approx_dashed}

\draw[dark-blue, very thick, rounded corners=0.5pt](5,2)--(7.75,2)--(7.5,5)--(5,5.75)--(3.25,3.75)--cycle;

\node[star, star points=5, draw, star point ratio=2.5, scale=0.4, fill=yellow] () at (5,2){};
\end{tikzpicture}
         \caption{Once enough cuts have been added to the approximation $ \widecheck{Q}_2$, the solution found cannot be separated from $Q_2$ and is thus optimal.}
     \end{subfigure}
    \caption{Geometric illustration of the Fenchel decomposition}
    \label{Fenchel_illustration}
\end{figure}

In the second step, in order to optimize over $LR_1 \cap \widecheck{Q}_2$, the following Fenchel master problem is used:
\begin{subequations}
\begin{alignat*}{3}
(F) \quad &\max_{x} && c^{T}x  \\
&\text{subject to} \quad && A_1 x \leq b_1 \\
&&& \pi x \leq \pi_0 \quad \forall (\pi, \pi_0) \in \mathcal{C}\\
&&& x \in \bar{X}
\end{alignat*}
\label{master_cut}
\end{subequations}
where $\mathcal{C}$ is the set of cuts describing $\widecheck{Q}_2$. 

The main challenge in the Fenchel decomposition is to generate a cut $ \pi x \leq \pi_0 $ separating the solution of the Fenchel master problem $ \widecheck{x} $ from the polyhedron $Q_2$. The classical approach to generating such cuts is based on a linear program as described in Section \ref{Fenchel_sp_normalization}.

\section{The Fenchel separation subproblem and its normalizations}
\label{Fenchel_sp_normalization}

In the Fenchel decomposition, a cut $ \pi x \leq \pi_0 $ separating the solution of the Fenchel master problem $ \widecheck{x} $ from the polyhedron $Q_2$ must be found. Such a cut can be created by finding a solution of non-negative value of the following separation linear program:
\begin{subequations}
\begin{alignat*}{3}
(S) \quad &\max_{\pi, \pi_0} && \pi \widecheck{x} - \pi_0 \\
&\text{subject to}  \quad && \pi x_i \leq \pi_0, \quad \forall x_i \in Q_2 \\
&&& \pi, \pi_0 \in \mathbb{R}
\end{alignat*}
\end{subequations}
where the objective maximizes the violation of the generated cut by $ \widecheck{x} $ while the constraints ensure that the cut is valid for every point of the polyhedron $ Q_2 $. However, this separation problem has too many constraints to consider them all explicitly. It is thus initially solved with a subset of its constraints ensuring that the generated cut is valid for a few points of $Q_2$. New points will progressively be taken into account in the constraints until the validity of the cut for the whole polyhedron can be ensured. To check if a cut $\bar{\pi} x \leq \bar{\pi_0}$ is valid for $Q_2$, we search for the point of $ Q_2 $ that most violates the cut. This can be done with a call to the optimization oracle ($ O $) for the linear function associated with $ \bar{\pi} $. If the point returned by ($O$) violates the cut, it is added as a constraint to ($S$). Otherwise, the cut $\bar{\pi} x \leq \bar{\pi_0}$ is the optimal solution of ($S$) and maximizes the separation of the solution of the Fenchel master problem from $Q_2$.

Another issue preventing the direct resolution of the problem ($ S $) is that its solution space is a cone. Indeed, if a cut $ \pi x \leq \pi_0 $ is valid for all the points of $ Q_2 $ then so is the cut $ \alpha \pi x \leq \alpha \pi_0 $ for any non-negative constant $ \alpha $. Thus, the problem ($S$) is often unbounded. To prevent this, a normalization process must be performed which classically consists in adding a constraint to the separation problem ($S$) that will make the set of possible coefficients of the generated cut bounded. The choice of this normalization greatly impacts the generated cut as well as the convergence speed of the Fenchel decomposition. In addition, the normalization may have an effect on the quality of the cut obtained, for instance by favoring (or not) the obtention of facets of the polyhedron $ Q_2 $. As we will see in the following sections, the normalization also impacts the dual problem of ($S$) which becomes ``\textit{Find the point of $Q_2$ minimizing some criterion of proximity to} $\widecheck{x}$". It is often more intuitive to interpret the impact of a normalization on the dual and therefore we use this approach for our geometric interpretations.

\subsection{Normalization $ \| \pi \| \leq 1 $}

\citet{boyd1995convergence} analyzes the normalization consisting in adding to the separation problem the constraint $ \| \pi \| \leq 1 $ for any norm  (\textit{i.e.} a function that satisfies the triangular inequality and for which for every $x \in \mathbb{R}^n$ and $\lambda \in \mathbb{R}$ satisfies $\|\lambda x\| = |\lambda| \| x\|$ and $\| x\| = 0 \Rightarrow x=0$). With this choice, we can consider that the optimized quantity is $ \frac{\pi \widecheck{x} - \pi_0}{\| \pi \|} $, \textit{i.e.} the distance from $ \widecheck{x} $ to the hyperplane $ \pi x = \pi_0 $. Note that this distance is not measured in the norm $ \|. \| $ but in its dual norm $ \|. \|^* $: $ \| \lambda \|^* = \ max_{\| x \| \leq 1} \lambda x $. In particular, note that the norms $ \|. \|_1 $ and $ \|. \|_{\infty} $ are dual and that the norm $ \|. \|_2 $ is self-dual. When using this normalization, the dual problem ($D$) becomes "Find the point of $Q_2$ closest to $\widecheck{x}$ in the sense of the dual norm $ \|. \|^* $". This point is located on the border of $ Q_2 $ and corresponds to the point of contact between $ Q_2 $ and the smallest sphere centered in $ \widecheck{x} $ touching $ Q_2 $. The generated cut is then a tangent cut to $ Q_2 $ and to the sphere at the point of contact. One important result of the work of \citet{boyd1995convergence} is that the polyhedron $ Q_2 $ can be computed in finite time with this normalization. On the other hand, the generated cuts are not always facets of the polyhedron $ Q_2 $. All these geometric interpretations are illustrated in the $ \|. \|_2 $ norm in Figure \ref{normalization_norme_2}.

\begin{figure}[ht]
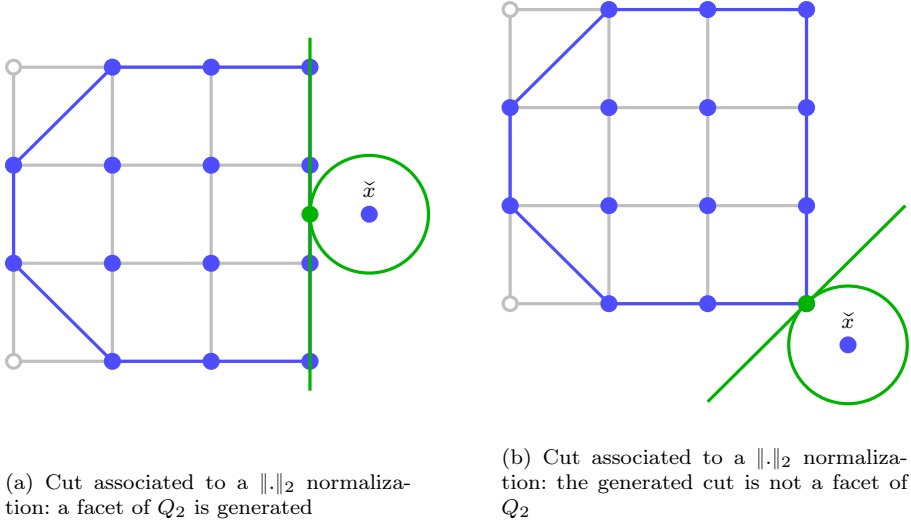

     \centering
     \begin{subfigure}[b]{0.45\columnwidth}
         \centering
         \begin{tikzpicture}[scale=1.3]
\input{figure_tikz/tikz_settings}
\input{figure_tikz/Components/grid_small}
\input{figure_tikz/Components/poly_approx_alone}

\draw[my-green, very thick, rounded corners=0.5pt](4,0.7)--(4,4.3);

\draw[dark-blue, fill=dark-blue, very thick](4.6,2.5) circle(0.07cm);

\draw[my-green, very thick](4.6,2.5) circle(0.6);

\draw[my-green, fill=my-green, very thick](4,2.5) circle(0.07cm);

\node (1) at (4.6,2.75) {$\widecheck{x}$};

\draw(2.5,0.3);
\end{tikzpicture}
         \caption{Cut associated to a $\|.\|_2$ normalization: a facet of $Q_2$ is generated}
     \end{subfigure}
     \hfill
     \begin{subfigure}[b]{0.45\columnwidth}
         \centering
         \begin{tikzpicture}[scale=1.3]
\input{figure_tikz/tikz_settings}
\input{figure_tikz/Components/grid_small}
\input{figure_tikz/Components/poly_approx_alone}

\draw[my-green, very thick, rounded corners=0.5pt](3,0)--(5,2);

\draw[dark-blue, fill=dark-blue, very thick](4.42,0.58) circle(0.07cm);

\draw[my-green, very thick](4.42,0.58) circle(0.6);

\draw[my-green, fill=my-green, very thick](4,1) circle(0.07cm);
\node (1) at (4.42,0.83) {$\widecheck{x}$};
\end{tikzpicture}
         \caption{Cut associated to a $\|.\|_2$ normalization: the generated cut is not a facet of $Q_2$}
     \end{subfigure}
    \caption{Example of cuts generated with a $\|.\|_2$ normalization}
    \label{normalization_norme_2}
\end{figure}

\subsection{Normalizations guaranteeing the generation of facets}

We are now interested in normalizations guaranteeing the generation of facets of $ Q_2 $. To do this, we start by presenting a theorem linking the facets of $ Q_2 $ and the extreme rays of the cone of the solutions of the separation problem ($ S $). This theorem can be found in \citet{conforti2019facet} (Proposition 1). In this theorem, $cone(X)$ (resp. $vect(X)$) is the set of all conical (resp. linear) combinations of elements of a set $X$.

\begin{theorem}
Let P be a nonempty polyhedron. Let $ (\pi^j x = \pi_0^j)_{j \in J ^ =} $ be a non-redundant representation of the affine envelope of P and let $ (\pi^j x \leq \pi_0^j)_{j \in J^{\leq}} $ be the set of facets of P. Then the set of valid cuts for P is:
$$\mathcal{C}(P) = cone\left(\begin{pmatrix} 0 \\1 \end{pmatrix}, \begin{pmatrix} \pi^j \\ \pi_0^j \end{pmatrix}_{j \in J^{\leq}} \right) + vect\left(\begin{pmatrix} \pi^j \\ \pi_0^j \end{pmatrix}_{j \in J^=} \right) $$
In addition, all these vectors are necessary for the previous representation.
\end{theorem}

This theorem states that apart from the trivial ray $ (0,1) $, the extreme rays of the cone of the solutions of ($ S $) are associated with the facets of the polyhedron $ Q_2 $. When the polyhedron is not fully dimensional (\emph{i.e.} $ J^= \neq \emptyset $), the cuts corresponding to $ (\pi^j, \pi_0^j)_{j \in J^=} $ are the improper faces of $Q_2$. Let us suppose for the following that $ Q_2 $ is full-dimensional. This happens once all the improper faces have been generated.

A sufficient condition ensuring that the cuts generated are facets of $ Q_2 $ is to use a normalization condition which can be applied by adding a unique linear constraint to the separation problem. If only one linear constraint is added to ($ S $), then all the extreme points of the new solution space correspond to old extreme rays and therefore to facets of $ Q_2 $. If the problem ($ S $) is indeed made bounded by the addition of the linear constraint, then it suffices to find an optimal vertex to guarantee the construction of a facet. More generally, it is possible to add several linear constraints, as long as they do not intersect inside the cone of solution of ($ S $). For example, adding a constraint $ -1 \leq f (\pi, \pi_0) \leq 1 $, where $ f $ is a linear function, corresponds to adding two linear constraints that do not intersect. The cuts generated will then induce facets of $ Q_2 $. When, for reasons specific to the problem, we know that valid cuts satisfy $ \pi \geq 0 $ (for example if $ Q_2 $ is a knapsack polyhedron), then the condition $ \| \pi \|_1 \leq 1 $ reduces to $ \sum_i \pi_i \leq 1 $ and can be used to generate facet inducing inequalities.

\textbf{Normalization of $ \pi_0 $:} A normalization guaranteeing the generation of facets is $ | \pi_0 | \leq $ 1. This normalization is studied in detail in \citet{conforti2019facet}. It makes the problem ($ S $) bounded if and only if $ \widecheck{x} $ can be written as a combination of elements in $ Q_2 $ using only non-negative coefficients. Indeed, let us look at the impact of normalization on the dual problem ($ D $) which becomes:
\begin{subequations}
\begin{alignat*}{3}
(D) \quad &\min_{\lambda_i, ~ z} && |z|  \\
&\text{subject to} \quad && \sum_{i \in I} \lambda_i x_i = \widecheck{x}\\
&&& \sum_{i \in I} \lambda_i = 1 + z\\
&&&  z \in \mathbb{R}, ~ \lambda_i \in \mathbb{R}^+, \quad \forall i \in I
\end{alignat*}
\end{subequations}
In simple words, this problem can be interpreted as follows: find a combination with non-negative coefficients of elements of $ Q_2 $ equal to $ \widecheck{x} $ whose sum of the coefficients is as close as possible to 1. The separation problem ($ S $) is bounded if and only if its dual ($ D $) is feasible. However, the problem ($ D $) is feasible if and only if $ \widecheck{x} $ can be written as a combination with non-negative coefficients of elements of $ Q_2 $. This condition is naturally reached when the origin is in the interior of $ Q_2 $ since the set of combinations with non-negative coefficients of elements of $ Q_2 $ is then the entire space. If we know a point in the interior of $ Q_2 $, it is possible to ensure this condition by translating the problem to place the origin on this interior point. Once the origin is inside $ Q_2 $, the generated cut is a facet of $ Q_2 $ intersecting the segment connecting $ \widecheck{x} $ to the origin.

\textbf{Directional normalization of $ \pi $:} A second normalization, studied by \citet{bonami2003etude} and guaranteeing the generation of facets, consists in bounding the coefficients of $ \pi $ in a given direction using the constraint $ | (\widehat{x} - \widecheck{x}) \pi | \leq 1 $ where $ \widehat{x} $ is an arbitrary point in $\mathbb{R}^n$. In order to determine when this normalization makes the problem ($ S $) bounded, let us look at its impact on the dual problem ($ D $) which becomes:
\begin{subequations}
\begin{alignat*}{3}
(D) \quad &\min_{\lambda_i, ~ z} && |z|  \\
&\text{subject to} \quad && \sum_{i \in I} \lambda_i x_i = z \widehat{x} + (1 - z) \widecheck{x}\\
&&& \sum_{i \in I} \lambda_i = 1\\
&&&  z \in \mathbb{R}, ~ \lambda_i \in \mathbb{R}^+, \quad \forall i \in I
\end{alignat*}
\end{subequations}
This problem can be interpreted as follows: \textit{find the point of $ Q_2 $ closest to $ \widecheck{x} $ on the line containing $ \widehat{x} $ and $ \widecheck{x} $}. The problem ($ S $) is bounded when its dual ($ D $) is feasible which happens when there is a point belonging to both $ Q_2 $ and to the previous line. This condition can be fulfilled for instance by choosing $ \widehat{x} $ equal to a known point of $ Q_2 $. In this case, the generated cut is a facet of $ Q_2 $ intersecting the segment connecting $ \widecheck{x} $ and $ \widehat{x} $. This is illustrated in Figure \ref{normalization_directionnelle_illustration}. Unlike the previous normalization, the known point $ \widehat{x} $ does not need to be in the interior of $ Q_2 $. However, note that when it is on the border of $ Q_2 $, the optimal value of the problem ($ S $) is bounded but the set of optimal solutions may be unbounded. Indeed, let us consider the following example illustrated in Figure \ref{normalization_directionnelle_cas_limite}.

\textit{Example: Suppose that $ \widehat{x} $ is located on a vertex of $ Q_2 $ and that $ \widecheck{x} $ is located such that the line $ (\widecheck{x}, \widehat{x}) $ intersects $ Q_2 $ only in $ \widehat{x} $. Let $ \pi x \leq \pi \widehat{x} $ be an optimal cut separating $ \widecheck{x} $ from $ Q_2 $ thus passing through $ \widehat{x} $. Let $ \pi^{\bot} $ be the projection of $ \pi $ on the orthogonal of the vector space induced by the vector $ \widecheck{x} - \widehat{x} $. Finally let us assume that the cut $ \pi^{\bot} x \leq \pi^{\bot} \widehat{x} $ is valid for $Q_2$. This is for instance the case in the 2D example illustrated in Figure \ref{normalization_directionnelle_cas_limite} as the hyperplane $ \pi^{\bot} x = \pi^{\bot} \widehat{x} $ is, in this case, the line $ (\widecheck{x}, \widehat{x}) $. Then for all $ \alpha \geq 0 $, the cut $ (\pi + \alpha \pi^{\bot}) x \leq (\pi + \alpha \pi^{\bot}) \widehat{x} $ is also an optimal cut for the separation problem. Indeed, it is valid for $Q_2$ as a non-negative combination of valid cuts for $Q_2$. Moreover, the violation of this cuts by $ \widecheck{x} $ is the same as the violation of $ \pi x \leq \pi \widehat{x} $ because by construction of $ \pi^{\bot} $ the product $ \pi^{\bot} ( \widecheck{x} - \widehat{x}) $ is equal to zero. Thus, in this example, the separation problem is bounded because its dual is feasible but the set of optimal cuts is unbounded because the norm of $\pi + \alpha \pi^{\bot}$ approaches infinity when $\alpha$ does the same.}

Despite this unbounded set of solutions, only the facets of the polyhedron $ Q_2 $ are vertices of the solution space of ($ S $). Thus, if the algorithm solving ($ S $) always returns a vertex, the generated cut will always be a facet of $Q_2$.

\begin{figure}[ht]
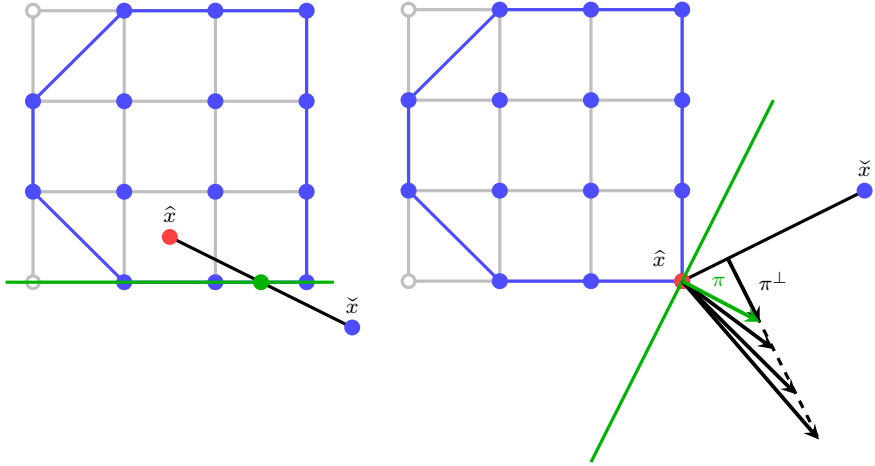

     \centering
     \begin{subfigure}[t]{0.43 \columnwidth}
         \centering
         \begin{tikzpicture}[scale=1.2]
\input{figure_tikz/tikz_settings}
\input{figure_tikz/Components/grid_small}
\input{figure_tikz/Components/poly_approx_alone}

\draw[black, very thick, rounded corners=0.5pt](4.5,0.5)--(2.5,1.5);

\draw[dark-blue, fill=dark-blue, very thick](4.5,0.5) circle(0.07cm);
\draw[my-red, fill=my-red, very thick](2.5,1.5) circle(0.07cm);

\draw[my-green, very thick, rounded corners=0.5pt](0.7,1)--(4.3,1);

\draw[my-green, fill=my-green, very thick](3.5,1) circle(0.07cm);
\draw(3.5,-1);

\node (1) at (4.5,0.75) {$\widecheck{x}$};
\node (1) at (2.5,1.75) {$\widehat{x}$};
\end{tikzpicture}
         \caption{Cut associated with directional normalization}
        \label{normalization_directionnelle_illustration}
     \end{subfigure}
     \begin{subfigure}[t]{0.56 \columnwidth}
         \centering
         \begin{tikzpicture}[scale=1.2]
\input{figure_tikz/tikz_settings}
\input{figure_tikz/Components/grid_small}
\input{figure_tikz/Components/poly_approx_alone}

\draw[black, very thick, rounded corners=0.5pt](6,2)--(4,1);
\draw[dark-blue, fill=dark-blue, very thick](6,2) circle(0.07cm);
\draw[my-red, fill=my-red, very thick](4,1) circle(0.07cm);

\draw[my-green, very thick, rounded corners=0.5pt](3,-1)--(5,3);

\draw[black, very thick, dashed, rounded corners=0.5pt](4.5,1.25)--(5.5,-0.75);

\draw[-stealth, line width=0.5mm] (4.5,1.25)--(4.85,0.55);
\draw[-stealth, line width=0.5mm] (4,1)--(5,0.25);
\draw[-stealth, line width=0.5mm] (4,1)--(5.25,-0.25);
\draw[-stealth, line width=0.5mm] (4,1)--(5.5,-0.75);
\draw[-stealth, my-green, line width=0.5mm] (4,1)--(4.85,0.55);

\node[draw=none, my-green] at (4.4,1) {$\pi$};
\node[draw=none] at (5,1) {$\pi^\bot$};
\node (1) at (6,2.25) {$\widecheck{x}$};
\node (1) at (3.75,1.25) {$\widehat{x}$};

\end{tikzpicture}
         \caption{Different normals of optimal cuts for the directional normalization: the norm of optimal normals can tend to infinity}
        \label{normalization_directionnelle_cas_limite}
     \end{subfigure}
    \caption{Examples of results obtained when using directional normalization.}
    \label{normalization_directionnelle}
\end{figure}

The above normalizations have been presented several times in the literature and are applicable in the general case. However, certain problems can admit normalizations particularly adapted to their structure. One can find such normalizations for example in the separation problems of disjunctive programming \citep{balas2002lift}. In addition, the unsplittable flow problem, which will be used in our experimental study, admits a normalization that seems natural and which be presented in Section \ref{normalization_naturelle}.

\section{A new approach for the Fenchel sub-problem}
\label{iterative_sub_problem}

In this section, we present a new approach to solve the separation problem of the Fenchel decomposition when the directional normalization presented in Section \ref{Fenchel_sp_normalization} is used. 

\subsection{Presentation of the method}

The proposed method is described in Algorithm \ref{algo:Fenchel_sp} and illustrated in Figure \ref{illustration_sub_problem}. Note that the cut associated with a directional normalization toward $ \widehat{x} $ is the same for $ \widecheck{x} $ than for any other point of the segment $ (\widecheck{x}, \widehat{x}) $ not belonging to $ Q_2 $.  The underlying idea is to compute intermediate cuts using an alternative normalization. By projecting a point $ x '$ (initially equal to $ \widecheck{x} $) onto these intermediate cuts, the algorithm gradually shifts the point $ x '$ along the segment $ (\widecheck{x}, \widehat{x}) $ in the direction of $ \widehat{x} $. Once $ x '$ reaches the frontier of $ Q_2 $ the procedure stops. Indeed, the last cut generated is the one associated with the directional normalization toward $\widehat{x}$ since it contains the intersection between the segment $ (\widecheck{x}, \widehat{x}) $ and the frontier of $Q_2$.

\begin{algorithm}
    \caption{New approach for solving the Fenchel sub-problem}
    \begin{algorithmic}[1]
    \Require a polyhedron $Q_2$, a point $\widecheck{x}$ to separate from $Q_2$, a point $\widehat{x}$ belonging to $Q_2$
    \Ensure a cut $C$ separating $\widecheck{x}$ from $Q_2$ and containing the intersection point the frontier of $Q_2$ and the segment $(\widecheck{x}, \widehat{x})$, a list $LS$ of vertices of $Q_2$ satisfying the cut $C$ to equality
    \State Set $x'$ equal to $\widecheck{x}$
    \While{$x' \notin Q_2$}
    \State $C$, $LS$ = Secondary\_separation($Q_2$, $x'$)
    \State Set $x'$ equal to the intersection point between the segment  $(\widecheck{x}$, $\widehat{x})$ and the cut $C$
    \EndWhile
    \State \Return $C$, $LS$
    \end{algorithmic}
    \label{algo:Fenchel_sp}
\end{algorithm}

\begin{figure}[ht]
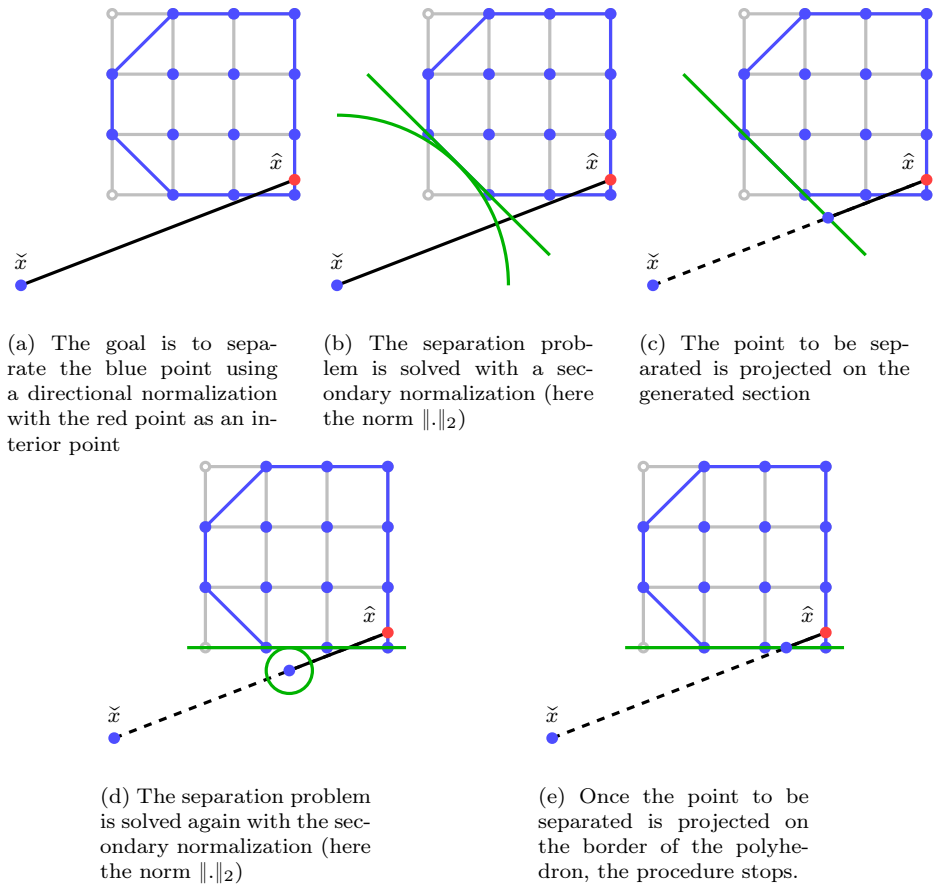

     \centering
     \begin{subfigure}[t]{0.3 \columnwidth}
         \centering
         \begin{tikzpicture}[scale=0.8]
\input{figure_tikz/tikz_settings}
\input{figure_tikz/Components/grid_small}
\input{figure_tikz/Components/poly_approx_alone}

\draw[black, very thick, rounded corners=0.5pt](-0.5,-0.5)--(4,1.25);

\draw[dark-blue, fill=dark-blue, very thick](-0.5,-0.5) circle(0.07cm);
\draw[my-red, fill=my-red, very thick](4,1.25) circle(0.07cm);
\node (1) at (-0.5,-0.1) {$\widecheck{x}$};
\node (1) at (3.7, 1.55) {$\widehat{x}$};

\end{tikzpicture}
         \caption{The goal is to separate the blue point using a directional normalization with the red point as an interior point}
     \end{subfigure}
     \hfill
     \begin{subfigure}[t]{0.3 \columnwidth}
         \centering
         \begin{tikzpicture}[scale=0.8]
\input{figure_tikz/tikz_settings}
\input{figure_tikz/Components/grid_small}
\input{figure_tikz/Components/poly_approx_alone}

\draw[black, very thick, rounded corners=0.5pt](-0.5,-0.5)--(4,1.25);

\draw[dark-blue, fill=dark-blue, very thick](-0.5,-0.5) circle(0.07cm);
\draw[my-red, fill=my-red, very thick](4,1.25) circle(0.07cm);

\draw[my-green, very thick, rounded corners=0.5pt](0,3)--(3,0);
\draw [my-green, very thick](2.32,-0.5) arc(0:90:2.82);
\node (1) at (-0.5,-0.1) {$\widecheck{x}$};
\node (1) at (3.7, 1.55) {$\widehat{x}$};

\end{tikzpicture}
         \caption{The separation problem is solved with a secondary normalization (here the norm $\|.\|_2$)}
     \end{subfigure}
     \hfill
     \begin{subfigure}[t]{0.3 \columnwidth}
         \centering
         \begin{tikzpicture}[scale=0.8]
\input{figure_tikz/tikz_settings}
\input{figure_tikz/Components/grid_small}
\input{figure_tikz/Components/poly_approx_alone}

\draw[black, very thick, dashed, rounded corners=0.5pt](-0.5,-0.5)--(4,1.25);
\draw[black, very thick, rounded corners=0.5pt](-0.5 + 2.88,-0.5 + 1.12)--(4,1.25);

\draw[dark-blue, fill=dark-blue, very thick](-0.5,-0.5) circle(0.07cm);
\draw[my-red, fill=my-red, very thick](4,1.25) circle(0.07cm);

\draw[my-green, very thick, rounded corners=0.5pt](0,3)--(3,0);

\draw[dark-blue, fill=dark-blue, very thick](-0.5 + 2.88,-0.5 + 1.12) circle(0.07cm);
\node (1) at (-0.5,-0.1) {$\widecheck{x}$};
\node (1) at (3.7, 1.55) {$\widehat{x}$};

\end{tikzpicture}
         \caption{The point to be separated is projected on the generated section}
     \end{subfigure}
     \hfill
     \begin{subfigure}[t]{0.3 \columnwidth}
         \centering
         \begin{tikzpicture}[scale=0.8]
\input{figure_tikz/tikz_settings}
\input{figure_tikz/Components/grid_small}
\input{figure_tikz/Components/poly_approx_alone}

\draw[black, very thick, dashed, rounded corners=0.5pt](-0.5,-0.5)--(4,1.25);
\draw[black, very thick, rounded corners=0.5pt](-0.5 + 2.88,-0.5 + 1.12)--(4,1.25);

\draw[dark-blue, fill=dark-blue, very thick](-0.5,-0.5) circle(0.07cm);
\draw[my-red, fill=my-red, very thick](4,1.25) circle(0.07cm);

\draw[my-green, very thick, rounded corners=0.5pt](0.7,1)--(4.3,1);

\draw[dark-blue, fill=dark-blue, very thick](-0.5 + 2.88,-0.5 + 1.12) circle(0.07cm);

\draw[my-green, very thick](-0.5 + 2.88,-0.5 + 1.12) circle(1.5-1.12);
\node (1) at (-0.5,-0.1) {$\widecheck{x}$};
\node (1) at (3.7, 1.55) {$\widehat{x}$};

\end{tikzpicture}
         \caption{The separation problem is solved again with the secondary normalization (here the norm $\|.\|_2$)}
     \end{subfigure}
     \hspace*{2cm}
     \begin{subfigure}[t]{0.3 \columnwidth}
         \centering
         \begin{tikzpicture}[scale=0.8]
\input{figure_tikz/tikz_settings}
\input{figure_tikz/Components/grid_small}
\input{figure_tikz/Components/poly_approx_alone}

\draw[black, very thick, dashed, rounded corners=0.5pt](-0.5,-0.5)--(4,1.25);
\draw[black, very thick, rounded corners=0.5pt](3.35,1)--(4,1.25);

\draw[dark-blue, fill=dark-blue, very thick](-0.5,-0.5) circle(0.07cm);
\draw[my-red, fill=my-red, very thick](4,1.25) circle(0.07cm);

\draw[my-green, very thick, rounded corners=0.5pt](0.7,1)--(4.3,1);

\draw[dark-blue, fill=dark-blue, very thick](3.35,1) circle(0.07cm);

\node (1) at (-0.5,-0.1) {$\widecheck{x}$};
\node (1) at (3.7, 1.55) {$\widehat{x}$};

\end{tikzpicture}
         \caption{Once the point to be separated is projected on the border of the polyhedron, the procedure stops.}
     \end{subfigure}
     \hfill
     \caption{Example of solving the Fenchel sub-problem for directional normalization using a secondary normalization}
    \label{illustration_sub_problem}
\end{figure}

The advantage of this iterative approach is that the linear separation program ($ S $) is never directly solved with the directional normalization, which often presents numerical instabilities. Indeed, for this normalization, the separation problem ($ S $) seeks a face of the polyhedron $ Q_2 $ intersecting the segment $ (\widecheck{x}, \widehat{x}) $. If, for example, this segment intersected a facet of $ Q_2 $ by forming an angle close to $ 0 $ with it then a small error on the parameters of the segment or of the facet can induce a large error on the position of the point of intersection. For this reason, the problem ($ S $) is sometimes too numerically unstable to be solved directly. The new method presented above yields an alternative way of computing the cut associated with the directional normalization without ever solving the numerically unstable linear program of the direct method. Unfortunately, note that this iterative algorithm requires the computation of the intersection between the segment $ (\widecheck{x}, \widehat{x}) $ and the cuts returned by the secondary separation problem. If the segment is almost parallel to one of these cuts, the computation of this point of intersection may still be numerically unstable. However, this is only an intermediate point not returned by the separation oracle. The results of the method are a cut and vertices which are computed with a secondary normalization. In practice, this seems sufficient as no numerical instabilities were found during our thorough numerical campaign.

Our experimental section yields insights into the practical performance of this new approach for the separation oracle. From a theoretical standpoint, we provide proofs in \ref{proof_subproblem} that the method converges in a finite number of iterations in two cases: 1) when the secondary normalization always generates facets of $ Q_2 $, and 2) when the secondary normalization is $ \| \pi \| \leq 1 $ for any norm. These two cases cover most normalizations used in the literature including all those mentioned in Section \ref{Fenchel_sp_normalization}.

\section{Coupling the Fenchel and Dantzig-Wolfe decompositions}
\label{hybrid_DW_Fenchel}

In this section, we present a decomposition method that integrates Dantzig-Wolfe and Fenchel decompositions. In our experimental campaign, we show that the new method presents a much superior performance than Fenchel decomposition alone, and is competitive against Dantzig-Wolfe on non-degenerate problems. For degenerate problems, again the proposed method shows a superior perormance to a classical Dantzig-Wolfe decomposition.

\subsection{Presentation of the method}

When generating a Fenchel cut, the primal variables of the separation sub-problem ($ S $) are the coefficients of the generated cut and the active constraints correspond to the vertices of the separated polyhedron $ Q_2 $ verifying this cut at equality. Thus, the Fenchel sub-problem generates both valid cuts for the polyhedron $ Q_2 $ and vertices of this polyhedron.

The main idea of the new decomposition method is to use two master problems operating in tandem. The first corresponds to the Fenchel formulation ($ F $) in which the generated Fenchel cuts are added to improve an outer approximation of $Q_2$. The second one is the Dantzig-Wolfe formulation ($ DW $) in which the vertices generated in the separation sub-problem are added to grow an inner approximation of $Q_2$. One of the key points of the method is that points $ \widecheck{x} $ provided by the Fenchel master problem are separated using directional normalization. This normalization requires the knowledge of a point in the polyhedron $ Q_2 $, for which we use the solution $ \widehat{x} $ computed by the Dantzig-Wolfe master problem.

The steps of the algorithm, illustrated in Figure \ref{illustration_DWF}, are as follows:
\begin{enumerate}
    \item Initialize an inner approximation $\widehat{Q}_2$ of the polyhedron $Q_2$ as in the Dantzig-Wolfe decomposition and an outer approximation $\widecheck{Q}_2$ as in the Fenchel decomposition.
    \item Optimize over $LR_1 \cap \widecheck{Q}_2$ using a Fenchel master problem: a solution $ \widecheck{x} $ is obtained whose value is an upper bound of the problem.
    \item Optimize over $LR_1 \cap \widehat{Q}_2$ using a Dantzig-Wolfe master problem: a solution $ \widehat{x} $ is obtained whose value is a lower bound of the problem.
    \item If the two bounds are equal: end of the algorithm
    \item With the Fenchel separation problem ($ S $), separate the point $ \widecheck{x} $ from $ Q_2 $ using directional normalization with $ \widehat{x} $ as the interior point: a cut is obtained as well as vertices of $ Q_2 $.
    \item Add the cut to the outer approximation $\widecheck{Q}_2$ and the vertices to the inner approximation $\widehat{Q}_2$. Then, go to Step 2.
\end{enumerate}

\begin{figure}[ht]
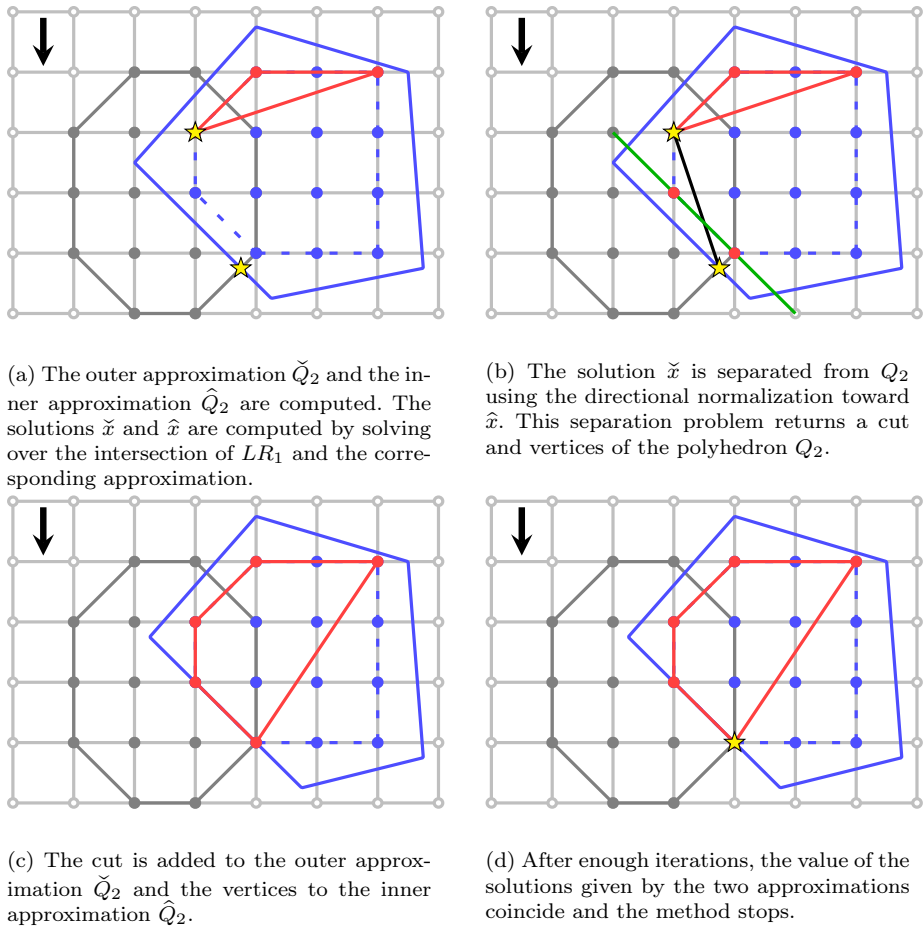

     \centering
     \begin{subfigure}[t]{0.47 \columnwidth}
         \centering
         \begin{tikzpicture}[scale=0.8]
\input{figure_tikz/tikz_settings}
\input{figure_tikz/Components/grid}
\input{figure_tikz/Components/poly_flots}
\input{figure_tikz/Components/poly_approx_dashed}

\draw[my-red, fill=my-red, very thick] (7,5) circle(0.07cm);
\draw[my-red, fill=my-red, very thick] (4,4) circle(0.07cm);
\draw[my-red, fill=my-red, very thick] (5,5) circle(0.07cm);
\draw[my-red, very thick, rounded corners=0.5pt](7,5)--(4,4)--(5,5)--cycle;

\draw[dark-blue, very thick, rounded corners=0.5pt](5.25,1.25)--(7.75,1.75)--(7.5,5)--(5,5.75)--(3,3.5)--cycle;

\node[star, star points=5, draw, star point ratio=2.5, scale=0.4, fill=yellow] () at (4,4){};
\node[star, star points=5, draw, star point ratio=2.5, scale=0.4, fill=yellow] () at (4.75,1.75){};
\end{tikzpicture}
         \caption{The outer approximation $\widecheck{Q}_2$ and the inner approximation $\widehat{Q}_2$ are computed. The solutions $ \widecheck{x} $ and $ \widehat{x} $ are computed by solving over the intersection of $LR_1$ and the corresponding approximation.}
     \end{subfigure}
     \hfill
     \begin{subfigure}[t]{0.47 \columnwidth}
         \centering
         \begin{tikzpicture}[scale=0.8]
\input{figure_tikz/tikz_settings}
\input{figure_tikz/Components/grid}
\input{figure_tikz/Components/poly_flots}
\input{figure_tikz/Components/poly_approx_dashed}

\draw[my-red, fill=my-red, very thick] (7,5) circle(0.07cm);
\draw[my-red, fill=my-red, very thick] (4,4) circle(0.07cm);
\draw[my-red, fill=my-red, very thick] (5,5) circle(0.07cm);
\draw[my-red, very thick, rounded corners=0.5pt](7,5)--(4,4)--(5,5)--cycle;

\draw[dark-blue, very thick, rounded corners=0.5pt](5.25,1.25)--(7.75,1.75)--(7.5,5)--(5,5.75)--(3,3.5)--cycle;

\draw[black, very thick, rounded corners=0.5pt](4,4)--(4.75,1.75);
\node[star, star points=5, draw, star point ratio=2.5, scale=0.4, fill=yellow] () at (4,4){};
\node[star, star points=5, draw, star point ratio=2.5, scale=0.4, fill=yellow] () at (4.75,1.75){};

\draw[my-green, very thick, rounded corners=0.5pt](3,4)--(6,1);
\draw[my-red, fill=my-red, very thick] (5,2) circle(0.07cm);
\draw[my-red, fill=my-red, very thick] (4,3) circle(0.07cm);
\end{tikzpicture}
         \caption{The solution $\widecheck{x}$ is separated from $Q_2$ using the directional normalization toward $\widehat{x}$. This separation problem returns a cut and vertices of the polyhedron $Q_2$.}
     \end{subfigure}
     \begin{subfigure}[t]{0.47 \columnwidth}
         \centering
         \begin{tikzpicture}[scale=0.8]
\input{figure_tikz/tikz_settings}
\input{figure_tikz/Components/grid}
\input{figure_tikz/Components/poly_flots}
\input{figure_tikz/Components/poly_approx_dashed}

\draw[dark-blue, very thick, rounded corners=0.5pt](5.75,1.25)--(7.75,1.75)--(7.5,5)--(5,5.75)--(3.25,3.75)--cycle;

\draw[my-red, fill=my-red, very thick] (7,5) circle(0.07cm);
\draw[my-red, fill=my-red, very thick] (4,4) circle(0.07cm);
\draw[my-red, fill=my-red, very thick] (5,5) circle(0.07cm);
\draw[my-red, fill=my-red, very thick] (5,2) circle(0.07cm);
\draw[my-red, fill=my-red, very thick] (4,3) circle(0.07cm);
\draw[my-red, very thick, rounded corners=0.5pt](7,5)--(5,5)--(4,4)--(4,3)--(5,2)--cycle;

\end{tikzpicture}
         \caption{The cut is added to the outer approximation $\widecheck{Q}_2$ and the vertices to the inner approximation $\widehat{Q}_2$.}
     \end{subfigure}
     \hfill
     \begin{subfigure}[t]{0.47 \columnwidth}
         \centering
         \begin{tikzpicture}[scale=0.8]
\input{figure_tikz/tikz_settings}
\input{figure_tikz/Components/grid}
\input{figure_tikz/Components/poly_flots}
\input{figure_tikz/Components/poly_approx_dashed}
\draw[dark-blue, very thick, rounded corners=0.5pt](5.75,1.25)--(7.75,1.75)--(7.5,5)--(5,5.75)--(3.25,3.75)--cycle;

\draw[my-red, fill=my-red, very thick] (7,5) circle(0.07cm);
\draw[my-red, fill=my-red, very thick] (4,4) circle(0.07cm);
\draw[my-red, fill=my-red, very thick] (5,5) circle(0.07cm);
\draw[my-red, fill=my-red, very thick] (5,2) circle(0.07cm);
\draw[my-red, fill=my-red, very thick] (4,3) circle(0.07cm);
\draw[my-red, very thick, rounded corners=0.5pt](7,5)--(5,5)--(4,4)--(4,3)--(5,2)--cycle;

\node[star, star points=5, draw, star point ratio=2.5, scale=0.4, fill=yellow] () at (5,2){};

\end{tikzpicture}
         \caption{After enough iterations, the value of the solutions given by the two approximations coincide and the method stops.}
     \end{subfigure}
     \caption{Illustration of the Dantzig-Wolfe-Fenchel decomposition}
    \label{illustration_DWF}
\end{figure}

Although this new decomposition method uses ideas taken from both the Dantzig-Wolfe and the Fenchel decompositions, the way in which they operate yields a few remarkable observations and interpretations that are worth describing:
\begin{itemize}
    \item Intuitively, the inner and outer approximations constructed disagree on the location of the boundary of $ Q_2 $ on the segment $ (\widecheck{x}, \widehat{x}) $. The separation problem finds the exact position of the frontier and gives information to the two approximations so that they can approximate exactly this part of the frontier.
    \item Another way of looking at the method is to say that it gradually improves the point $ \widehat{x} $. At each iteration, the algorithm tests a direction of potential improvement $ \widecheck{x} - \widehat{x} $. The method then finds either new vertices of $ Q_2 $ allowing to improve $ \widehat{x} $ or a facet of $ Q_2 $ passing through $ \widehat{x} $ proving that it is not possible to improve $ \widehat{x} $ in this direction. This interaction between an exterior point and an interior point of $ Q_2 $ is reminiscent of the \textit{in-out} separation proposed for the Benders decomposition \citep{ben2007acceleration}.
    \item Compared to a classic Fenchel decomposition, this method concentrates its cut generation around the point $ \widehat{x} $. It thus refines the knowledge of the polyhedron $ Q_2 $ around this point. On the other hand, a classical Fenchel decomposition might spend multiple iterations searching for cuts in regions of the solution space that end up being far from the optimum. 
    \item Compared to a classic Dantzig-Wolfe method, this method devotes more time to the resolution of its sub-problem which allows it to generate a greater number of vertices to add to the master problem.
\end{itemize}

\subsection{Degeneracy}

The Dantzig-Wolfe decomposition is known to present convergence issues when its master problem is highly degenerate. Our experiments reveal that the method presented in this section does not suffer from the same issues. In the following, we give a partial theoretical explanation of the absence of degeneracy issues in the new method.

A linear program is said to be degenerate when it has multiple dual optima. We have seen in Section \ref{sec:presentation-DW} that each dual solution of the Dantzig-Wolfe master problem implies a bound on the value of its objective function $cx \leq c_0$ that certifies the optimality of its current primal solution on the solution set $LR_1 \cap \widehat{Q}_2$. Thus, in order to improve the solution of the master problem, one needs to be able to separate each of these dual solutions from the polyhedron $LR_1 \cap Q_2$. However, the sub-problem of the Dantzig-Wolfe decomposition may only separate one such dual vector which explains why the Dantzig-Wolfe decomposition is so affected by degeneracy.

The new proposed Dantzig-Wolfe-Fenchel decomposition seems to be unaffected by the degeneracy of its Dantzig-Wolfe master problem. One possible explanation of this behavior is supported in the following observation and proposition. First, let us remark that the new decomposition neither computes nor uses any dual information, therefore showing no sensitivity to the quality of the duals. Second:
\begin{proposition}
\label{claim_degeneracy}
At each iteration, the sub-problem invalidates either all the dual solutions of the Dantzig-Wolfe master program or none of them.
\end{proposition}

\begin{proof}
The sub-problem finds the farthest point in $Q_2$ along the segment $ (\widehat{x}, \widecheck{x}) $. On the one hand, if this farthest point coincides with $\widehat{x}$ then it does not invalidate any of the bounds implied by the dual solutions because $\widehat{x}$ belongs to $LR_1 \cap \widehat{Q}_2$. On the other hand, suppose the farthest point $x^*$ does not coincide with $\widehat{x}$. First, by construction, $x^*$ belongs to $Q_2$ but also to $LR_1$ because it is a convex combination of $\widehat{x}$ and $\widecheck{x}$, both belonging to $LR_1$. Second, assuming that $\widecheck{x}$ has a strictly better objective value than $\widehat{x}$ (which is always the case except when the method is about to terminate) then $x^*$ also has a strictly better objective value than $\widehat{x}$. Thus, it must invalidate all the bounds implied by the dual solutions of the Dantzig-Wolfe master problem (remember that these bounds certify that no better solution than $\widehat{x}$ exists). \qed
\end{proof}

Therefore, whether the Dantzig-Wolfe master problem admits multiple dual optima or not does not influence the algorithm's capacity to find a strictly improving point.

Although we presented all the decomposition methods as if only one polyhedron was decomposed at once (e.g. a block of a block diagonal matrix), in practice several polyhedra are decomposed at the same time (e.g. all the blocks of a block diagonal matrix). In this case, Proposition \ref{claim_degeneracy} does not hold. However, the  ideas discussed in its proof may still impact positively the practical computations. In any case, the decomposition method still does not use any dual information which renders it oblivious to the number of dual solutions of its Dantzig-Wolfe master problem.

\section{Application to the unsplittable flow problem}
\label{sec:application}

In the unsplittable flow problem (UFP), one is given a weighted directed graph $\mathcal{G} = (V, A, c)$ where $c_a$ is the capacity of the arc $a$ for every $a\in A$. We are also given a family $K$ of commodities, each composed of an origin $o(k)$, a destination $d(k)$, and a demand $D_k$, for every $k\in K$. Each commodity has to be routed from its origin to its destination through an unique path. We consider the problem where the capacity constraints are soft, meaning that they can be violated at a certain unit penalty. The objective is to design routing paths for every commodity on the network so as to minimize the sum of the violations of the arcs' capacities. 

The UFP is an extensively studied NP-hard variant of the classic maximum-flow problem. It has multiple applications, as for instance in telecommunication networks (\emph{e.g.} optical networks, telecommunication satellites \citep{rivano2002lightpath, lamothe2021dynamic}), and logistics \citep{farvolden1993primal}. The early work of \citet{belaidouni2007minimum} studied the UFP from a polyhedral perspective, proposing cutting planes to strengthen the linear relaxation of a three-index model that uses variables $x_{ijk}$ for every arc $(i, j)$ in the network and every commodity $k$ to be transported. \citet{park2003integer} also strengthen the linear relaxation by applying the Dantzig-Wolfe decomposition to the capacity constraints of the UFP. The resulting relaxation is as strong as if all the inequalities valid for the capacity constraints of the problem were added. Thus, faster decomposition methods able to compute the Dantzig-Wolfe linear relaxation could yield improvements in the resolution of the UFP.

We will consider an arc-path formulation where the meaning of the variables is the following:
\begin{itemize}
    \item $x_{pk}$ indicates whether commodity $k$ uses path $p$ to push its flow,
    \item $\Delta_a$ represents the overflow on arc $a$.
\end{itemize}

In addition to the decision variables, we also denote, for a given commodity $k\in K$, $P_k$ the set of all $o(k)$-$d(k)$-paths in $G$. For every $p\in P_k$ and arc $a\in A$, we define a constant $\alpha_{ap}$ that takes the value 1 iff $p$ uses the arc $a$. The path formulation of the UFP is the following:
\begin{subequations}
\begin{alignat}{4}
&\min_{x_{pk},\Delta_a} && \sum_{a \in A} \Delta_a & ~~ &  \\
&\text{subject to} ~~ && \sum_{p \in P_k} x_{pk} = 1 & & \forall k \in K, \label{eq: convex_p}\\
&&&\sum_{k \in K} \sum_{p \in P_k, a \in p} x_{pk} D_k \leq c_a + \Delta_a & & \forall a \in A,  \label{eq: capacity_const_p}\\
&&& x_{pk} \in \{0,1\}, ~ \Delta_a \in \mathbb{R}^+ & & \forall p \in P_k, ~ \forall k \in K, ~ \forall a \in A
\end{alignat}
\end{subequations}

The objective function minimizes the sum of the overflows on the arcs. Equation (\ref{eq: convex_p}) ensures that exactly one path is chosen for each commodity. Equation (\ref{eq: capacity_const_p}) corresponds to the soft capacity constraints. It ensures that any overflow on an arc $a$ is recorded on the corresponding variable $\Delta_a$. The fact that $ x_{pk} \in \{0,1\}$ ensures that the flow is unsplittable.

The polyhedron associated with the capacity constraints does not have the integrality property and its relaxation can thus be tightened with any of the previously discussed decomposition methods. If we denote $f_a^k = \sum_{p \in P_k, a \in p} x_{pk}$, this polyhedron can be written as follows:

$$ \left \{f_a^1, ..., f_a^k \in [0, 1], o_a \in \mathbb{R}^+ \bigg | \sum_{k \in K} f_a^k D^k \leq c_a + o_a \right \}. $$

Studies have been carried out on the structure, the cut selection, and the strengthening of the linear relaxation of this type of polyhedron by \citet{Marchand19990} as well as in the more general framework of linear programs in mixed variables \citep{dash2011mixed, fukasawa2011exact, chvatal2013local}. Moreover, optimization methods on this polyhedron have been studied by \citet{buther2012reducing, lin2011exact, zhao2014approximation, he2019encoding, liu2017exact}.

In the following, we present a specialized normalization for this polyhedron. This normalization guarantees the generation of facets and will be used in our implementation of the Fenchel decomposition and as secondary normalization in the new procedure to solve the Fenchel sub-problem presented in Section \ref{Fenchel_sp_normalization}. Then, we describe our implementation of the oracle that optimizes a linear function on the polyhedron associated with the capacity constraints of the UFP.

\subsection{Natural normalization for unsplittable flows}
\label{normalization_naturelle}

Let $ \widecheck{f} = (\widecheck{f}^k)_{k \in K} \in [0, 1] ^ {| K |} $ be a flow distribution for each commodity which induces an overflow $ \widecheck{o} $ on a given arc. In this section, we assume that the arc is fixed and will therefore drop the arc index for the sake of simplicity. In the context of unsplittable flows, vertices of the polyhedron $Q_2$ correspond to commodity patterns which will be indexed by a superscript $g$. A naturally occurring question is: how does this distribution break down into a combination of commodity patterns inducing a minimum capacity overflow? This question can be solved using the following linear program:
\begin{subequations}
\begin{alignat*}{3}
(D) \quad &\min_{\lambda^g, z} && z \\
&\text{subject to} \quad && \sum_{g \in G} \lambda^g f^g = \widecheck{f}\\
&&& \sum_{g \in G} \lambda^g o^g = \widecheck{o} + z\\
&&& \sum_{g \in G} \lambda^g = 1\\
&&& \lambda^g \in \mathbb{R}^+, ~ z \in \mathbb{R}^+ \quad \forall g \in G
\end{alignat*}
\end{subequations}
where $ \lambda^g $ is the coefficient in the decomposition associated with a commodity pattern $ f^g $ inducing an overflow $ o^g $.

Now the dual of this decomposition program is the following program:

\begin{subequations}
\begin{alignat*}{3}
(S) \quad &\max_{\pi, \pi_o \pi_0} && \pi \widecheck{f} + \pi_o \widecheck{o} - \pi_0  \\
&\text{subject to} \quad && \pi f^g + \pi_o o^g \leq \pi_0, \quad \forall (f^g, o^g) \in Q_2 \\
&&& \pi_o \leq -1 \\
&&& \pi, \pi_0, \pi_o \in \mathbb{R}
\end{alignat*}
\end{subequations}

This program corresponds exactly to the problem ($ S $) of separating the point $ \widecheck{x} = (\widecheck{f_1}, ..., \widecheck{f}_{|K|}, \widecheck{o}) $ with a constraint of normalization imposing that the coefficient $ \pi_o $ associated with the overflow variable satisfies $ \pi_o \leq -1 $. This normalization constraint is what we will call the natural normalization for the unsplittable flow problem. This normalization is very close to a particular case of directional normalization for the direction $ \widehat{x} - \widecheck{x} = (0, ..., 0, 1) $. Just like directional normalization, the natural normalization guarantees the generation facets because it is imposed using a single linear constraint.

\subsection{Knapsack oracle resolution}

All the decomposition methods presented in this work assume that there exists an efficient algorithm capable of optimizing a linear function on the polyhedron $ Q_2 $. In this section, we detail the problem solved by the oracle in the context of unsplittable flows.

In the version of the unsplittable flow problem that we are studying, the capacity constraints do not require that the flow of commodities respect the capacities $ c_a $ of the arcs. However, the overflow must be stored in a variable $ o_a $. Thus, the polyhedron of variables satisfying the soft capacity constraint associated with the arc $ a $ is written:
$$\left\{(f_a^k \in \{0, 1\})_{a \in A, k \in k}, o_a \in \mathbb{R}^+ \Bigg| \sum_{k \in K} f_a^k D^k \leq c_a + o_a\right\}.$$
The optimization of a linear function whose coefficients are $(\pi^k)_{k \in K}$ and $-\pi_o$ on this polyhedron can be written as follows:

\begin{subequations}
\begin{alignat*}{3}
(O_a) \quad &\max_{f_a^k, o_a} && \sum_{k \in K} \pi^k f_a^k - \pi_o o_a  \\
&\text{subject to} \quad && \sum_{k \in K} f_a^k D^k \leq c_a + o_a \\
&&& f_a^k \in \{0, 1\}, ~ o_a \in \mathbb{R}^+
\end{alignat*}
\end{subequations}

This problem can be solved as a sequence of two 0-1 knapsack problems using a case disjunction. This method was presented by \citet{buther2012reducing} and is recalled in \ref{annex_knapsack}. In our experiments, we use the MINKNAP algorithm proposed by \citet{pisinger1997minimal} to solve the two associated knapsack problems.

\section{Experimental study}
\label{sec:results}

In this section, we present an experimental comparison of different decomposition methods. The datasets and code used in this section can be accessed at \url{https://github.com/TwistedNerves/decomposition_paper_code}. The code was written in Python 3 and the experiments carried on an Intel Core i9-9900K 3.60 GHz $\times 16$ cores CPU, 60 Gbit of RAM, running Ubuntu 20.10.

\subsection{Datasets}

An instance of the unsplittable flow problem is composed of a graph and a list of commodities. The method used to create instances in our experiments is the one presented in \citet{lamothe2021randomized}. All the graphs used are strongly connected random graphs. To create demands for the commodities, \citet{lamothe2021randomized} used two formulas. In this work, we use the formula that creates mainly commodities with large demands because it tends to create instances that are harder to solve. Moreover, in each instance, all the commodities can be unsplittably routed without exceeding the arc capacities. Therefore, the lower bound given by the linear relaxation is optimal. In order to create an optimality gap in the instances, we slightly modify the capacities of some arcs as follows a number of times equal to 100 times the number of nodes:

\begin{itemize}
    \item Randomly select the origin of a commodity.
    \item Randomly select two arcs coming out of this origin.
    \item Add 1 to the capacity of one arc and subtract 1 from the ability of the other arc.
\end{itemize}

Because of the way instances are created, all outgoing arcs from each origin node are saturated in the solutions without overflow while the other arcs are often non-saturated. Therefore, in most cases, transferring some capacity between outgoing arcs of origins does not change the value of the linear relaxation. On the other hand, this transfer of capacity can have an impact on the value of the best unsplittable solution. Indeed, there is no longer necessarily a combination of commodities whose sum of demands is exactly equal to the capacity of each arc. In this case, the best unsplittable solution has a non-zero overflow.

Another change made to the instances is that the commodities have only access to a restricted set of paths to push their flow. The restricted set of paths of a commodity is chosen to be the k-shortest paths from the origin to the destination of the commodity with k = 4. Because this study explores the strengthening of the linear relaxation of the unsplittable flow problem through its capacity constraints, this modification should not change the relative behavior of the tested algorithms but does make the instances much simpler to solve which enables the testing of the different algorithms on larger instances.

\paragraph{The datasets} 

Three different datasets are used during the experiments in which ten instances are generated for each value of the varying parameter.
\begin{itemize}
    \item \textit{Low maximum demand} dataset: this dataset considers strongly connected random graphs from 50 nodes to 145 nodes. The maximum commodity demand is set at $ \widecheck{d}_{max} = $ 100 and the arc capacity at 1000. This choice of maximum demand implies that a large number of commodities can pass through each arch. However, in our tests, the optimal solution for these instances often does not contain overflow. Therefore, these instances do not contain an optimality gap. We hypothesize that the large number of commodities allows them to rearrange themselves to exactly fill the capacity of each arc.
    \item \textit{High maximum demand} dataset: this dataset considers strongly connected random graphs from 145 nodes to 1000 nodes. The maximum commodity demand is set at $ \widecheck{d}_{max} = $ 1000 and the arc capacity at 1000. Because of this maximum demand choice, these instances contain only a small number of commodities. However, they generally have an optimality gap which allows us to study the evolution of the lower bounds given by the algorithms.
    \item \textit{Size of capacities} dataset: this dataset considers strongly connected random graphs of 70 nodes. The maximum demand of the $ \widecheck{d}_{max} $ commodities is fixed at $ 1/10 $ of the common capacity of the arcs which varies from 100 to 100,000. The knapsack problem is known to have algorithms that are pseudo-polynomial in the capacity of the knapsack. One such algorithm is the MINKNAP algorithm we use. In the case of unsplittable flows, the capacity of the knapsack corresponds to the capacity of the arcs. The instances of this dataset all have the same structure (same graph size, same size of commodities relative to the capacity of the arcs) but varying arc capacities. This impacts the resolution time of the MINKNAP algorithm.
\end{itemize}

\subsection{Comparison of Dantzig-Wolfe variants}
\label{DW_comparison}

A large number of works have proposed variations and improvements to Dantzig-Wolfe decomposition. Indeed, the method presented in Section \ref{sec:presentation-DW} sometimes suffers from a slow convergence which has been associated with the following observations \citep{pessoa2013out}:
\begin{itemize}
    \item \textbf{Dual oscillations:} the dual variables $\pi$ used to generate the vertices of $ Q_2 $ perform large oscillations and do not converge monotonically toward their optimal value.
    \item \textbf{The tailing-off effect:} during the last iterations, the space of the dual solutions is only marginally reduced and the dual bound progresses very slowly.
    \item \textbf{Degenerate primal and equivalent dual solutions:} The master problem ($ DW $) is regularly degenerate because it has several dual optimal solutions. The method iterates between equivalent dual solutions without making progress on the value of the objective function.
\end{itemize}
In order to overcome these difficulties, stabilization methods for the dual variables have been considered. These methods can be classified into three categories \citep{pessoa2013out}:
\begin{itemize}
    \item \textbf{Penalization:} the penalization methods are best interpreted by considering the dual of the problem ($ DW $). In order to stabilize the dual variables $ \pi $, a penalty $ f (\| \pi - \bar{\pi} \|) $ is added to the objective function of the dual. In this penalty, $ f() $ is an increasing function, which pushes $ \pi $ to stay close to a value $ \bar{\pi} $ which evolves slowly during the algorithm. Typically, $ \bar{\pi} $ is one of the values taken by the dual variables during the previous iterations. A widely studied special case is to penalize proportionally to $ \| \pi - \bar{\pi} \|_2^2 $, which is done in the \textit{Bundle} methods \citep{briant2008comparison}.
    \item \textbf{Smoothing:} in smoothing methods, the dual variables $ \pi $ of the problem ($ DW $) are not used directly in the sub-problem in order to generate new vertices of $ Q_2 $. We note, for the iteration $ j $ of column generation, $ \pi_j $ the values of the variables resulting from the dual of the problem ($ DW $) and $ \bar{\pi}_j $ the values used in the sub-problem. A smoothing method proposed by \citet{neame2000nonsmooth} uses the following formula: $ \bar{\pi}_j = \alpha \bar{\pi}_{j-1} + (1 - \alpha) \pi_j $. This method amounts to adding an \textit{momentum} effect to the dual variables. Another method, proposed by \citet{wentges1997weighted}, performs a convex combination with a fixed dual value $ \bar{\pi} $, i.e. $ \bar{\pi}_j = \alpha \bar{\pi} + (1 - \alpha) \pi_j $.
    \item \textbf{Centralization:} The idea of centralization methods is that it is more efficient to use in the sub-problem dual values located inside the dual polyhedron rather than on an extreme vertex of the dual polyhedron. On the other hand, such interior points are more expensive to compute than extreme points. The interior point used in the Primal-Dual Column Generation \citep{gondzio2013new} is obtained by approximately solving the problem ($ DW $) by an interior point method. Another classic point is the analytical center used in the analytical center cutting plane method \citep{goffin2002convex}.
\end{itemize}

In order to have a suitable comparison for the decomposition methods we proposed in this work, we experimentally compared the following three variations of Dantzig-Wolfe decomposition:

\textbf{DW:} Dantzig-Wolfe decomposition method. No stabilization of the dual variables is used. The lower bounds are computed using the dual variables and the value of the solution of the knapsack sub-problems.

\textbf{DW-momentum:} similar to the \textit{DW} method except that the dual variables are stabilized by smoothing using the formula of \citet{neame2000nonsmooth}, $ \bar{\pi}_t = \alpha \bar{\pi}_{t-1} + (1 - \alpha) \pi_t $ with the coefficient $\alpha$ set to $ 0.8 $.

\textbf{DW-interior-point:} similar to the \textit{DW} method except that the Dantzig-Wolfe master problem is solved with an interior point method in order to return a non-optimal but centered solution. To that end, we ask the \citep{gurobi} solver to solve the linear program using an interior point method with a precision of $ 10^{- 3} $ and without using its \textit{crossover} method. However, the first time the sub-problem fails to generate a new negative reduced cost variable, the solver Gurobi is reset to its default settings to ensure an exact computation of the last reduced costs. With the default parameters, the generation of columns is no longer stabilized.

These methods are compared in Figure \ref{Compare_DW_70}. For the rest of the experiments, we will use the variation based on the interior point solver as it always returns the best results in our tests.

 \begin{figure}[ht]
\centering
\begin{minipage}{0.7\textwidth}
     \centering
     \includegraphics[width=\textwidth, trim={0 0 45 36},clip]{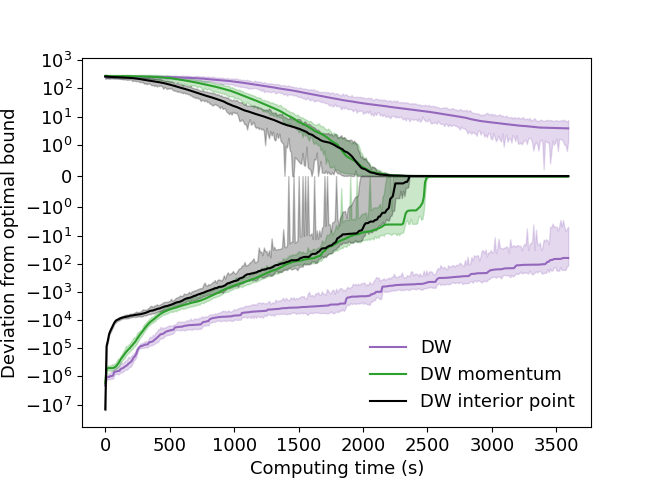}
     \caption{Compararison of Dantzig-Wolfe decomposition variations, \textit{Low maximum demand} dataset, 70 nodes}
     \label{Compare_DW_70}
\end{minipage}
 \end{figure}

\subsection{The decomposition methods studied}

In the following, we experimentally compare the following decomposition methods:

\textbf{Fenchel:} Fenchel decomposition method. The cuts generated are added to the linear relaxation while the generated vertices are added to a Dantzig-Wolfe formulation. The Fenchel sub-problem is solved with the natural normalization presented in Section \ref{normalization_naturelle}. Therefore, the two master problems do not act in tandem.

\textbf{DW-Fenchel:} method combining the Fenchel and Dantzig-Wolfe decompositions presented in Section \ref{hybrid_DW_Fenchel}, the cuts generated are added to the linear relaxation while the generated vertices are added to a Dantzig-Wolfe formulation. The Fenchel sub-problem is solved with directional normalization with the optimal point of the Dantzig-Wolfe formulation as the interior point. The use of this normalization couples the two formulations.

\textbf{DW-Fenchel-iterative:} similar to the \textbf{DW-Fenchel} method except that the Fenchel sub-problem is solved using the iterative method presented in Section \ref{iterative_sub_problem}.

\textbf{DW-interior-point:} Dantzig-Wolfe decomposition introduced in Section \ref{DW_comparison} where the master problem is solved with an interior point method in order to return a non-optimal but centered solution.

\subsection{Algorithms' parameters}

\textit{Authorized paths} Because the focus of this work is on the capacity constraints and not on how to generate the paths for each commodity, each commodity is restrained to a small set of allowed paths. This set is made up of the four shortest paths between the origin and destination of the commodity as well as the path used to create the commodity in the method of \citet{lamothe2021randomized}.

\textit{Algorithm termination condition} The decomposition methods considered are stopped when the absolute difference between their bounds is $ 10^{- 3} $.

\textit{Pre/post-processings for the sub-problem of Fenchel:} Solving directly a Fenchel subproblem is sometimes too computationally expensive to be integrated into a decomposition method. However, the resolution time of this subproblem can be greatly reduced with pre/post-processing steps. Indeed, \citet{boccia2008cut} showed that it is possible to solve the Fenchel separation problem by focusing on a sub-polyhedron of $Q_2$ of far lesser dimension which decreases the computing time. However, the generated cut is not directly valid for $Q_2$ and one must use a \textit{lifting} procedure to create a cut valid for $Q_2$. These concepts are explained in \ref{pre_post_traitement}. Together, the techniques of \textit{Dimensionality reduction} and \textit{Lifting} induce a drastic decrease in the resolution time of the Fenchel subproblem. This fact was confirmed by our experiments and the results we present in our experimental study reflect this algorithmic choice.

\subsection{Experimental results}

We now present the results of our experimental campaign for the different decomposition methods presented in this work. In each figure, we display the evolution of the lower and upper bounds achieved by the algorithms as a function of the computational time (in seconds). Note that all the displayed values are not directly the bounds but their deviation from the value of an optimal solution of the Dantzig-Wolfe reformulation. The plotted curves represent the average results of the algorithms aggregated on instances using the same parameters while the confidence intervals at $ 95 \% $ for the mean are plotted in semi-transparency around the main curve.
 A problem encountered when generating these curves is that the algorithms only return bounds at the end of each of their iterations but these iterations take a variable time for the same algorithm depending on the instance. It is therefore not possible to directly aggregate the curves using the points given at the end of each iteration because they do not correspond to the same computing time. To obtain points on which we can appropriately average the values, the points defining the curves are replaced with points sampled every ten seconds by considering that the bounds evolve linearly between two iterations. 
 The confidence intervals are created using the statistical method called Bootstrapping with a number of resamplings equal to 1000. Because the Bootstrapping method is applied independently for every ten seconds of the curves the resulting confidence intervals have jitters. These jitters can be interpreted as the uncertainty on the bound of the confidence intervals due to the Bootstrapping method.

\begin{figure}
\end{figure}

 \begin{figure}
\begin{minipage}{0.49\textwidth}
     \centering
     \includegraphics[width=0.7\textwidth]{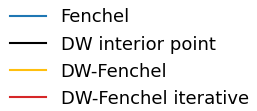}
     \caption{Legend of the Figures \ref{small_all_algo} to \ref{capacity_10000}}
     \label{Legende_decomposition}
\end{minipage}
\hfill
\begin{minipage}{0.49\textwidth}
     \centering
     \includegraphics[width=\textwidth, trim={0 0 45 36},clip]{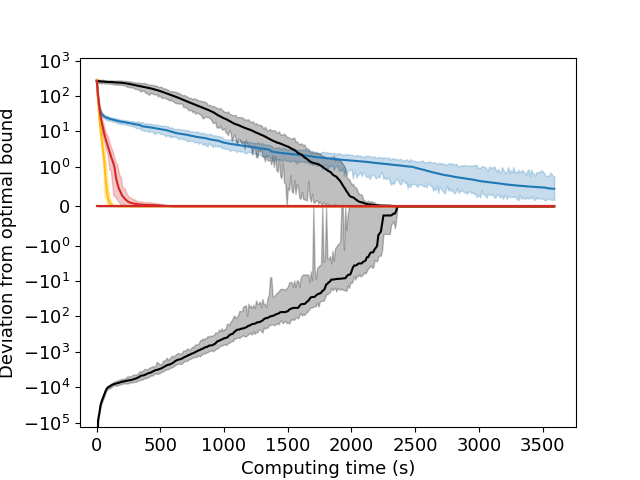}
     \caption{\textit{Low maximum demand} dataset, 70 nodes}
     \label{small_all_algo}
\end{minipage}
 \end{figure}

 \begin{figure}
\begin{minipage}{0.49\textwidth}
    \centering
    \includegraphics[width=\textwidth, trim={0 0 45 36},clip]{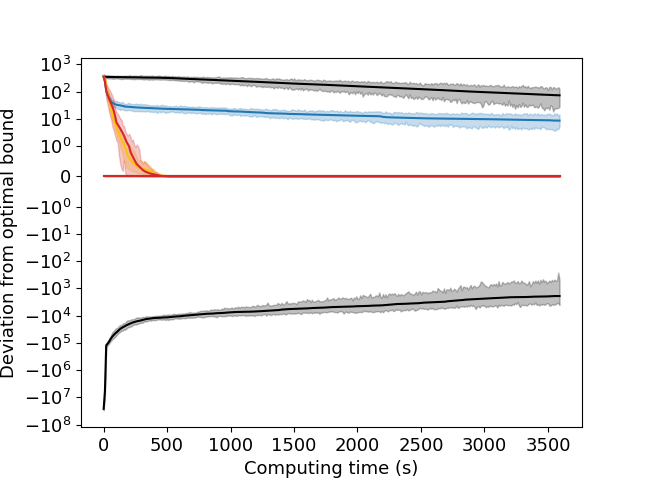}
    \caption{\textit{Low maximum demand} dataset, 90 nodes}
    \label{faible_demande_90}
\end{minipage}
\hfill
\begin{minipage}{0.49\textwidth}
     \centering
     \includegraphics[width=\textwidth, trim={0 0 45 36},clip]{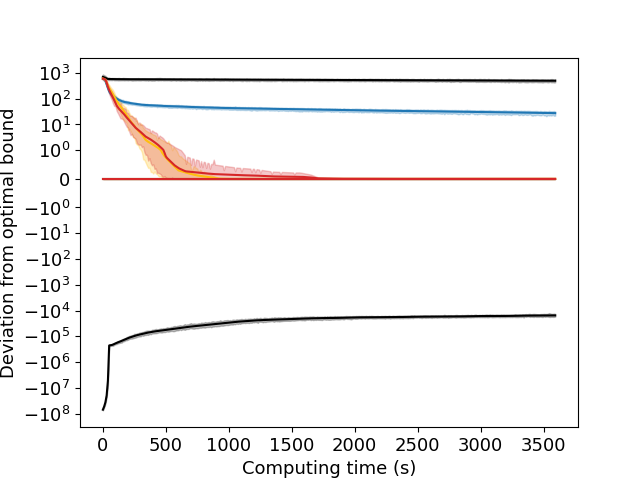}
     \caption{\textit{Low maximum demand} dataset, 145 nodes}
     \label{faible_demande_145}
\end{minipage}
 \end{figure}

\begin{figure}
\begin{minipage}{0.49\textwidth}
     \centering
     \includegraphics[width=\textwidth, trim={0 0 45 36},clip]{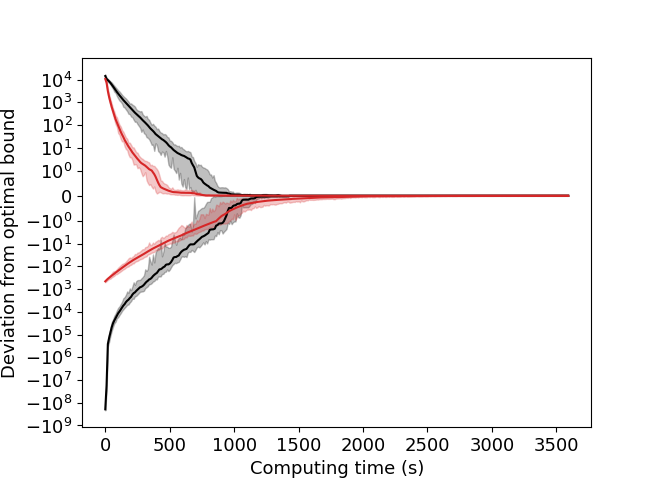}
     \caption{\textit{High maximum demand} dataset, 250 nodes}
     \label{forte_demande_250}
\end{minipage}
\hfill
\begin{minipage}{0.49\textwidth}
     \centering
     \includegraphics[width=\textwidth, trim={0 0 45 36},clip]{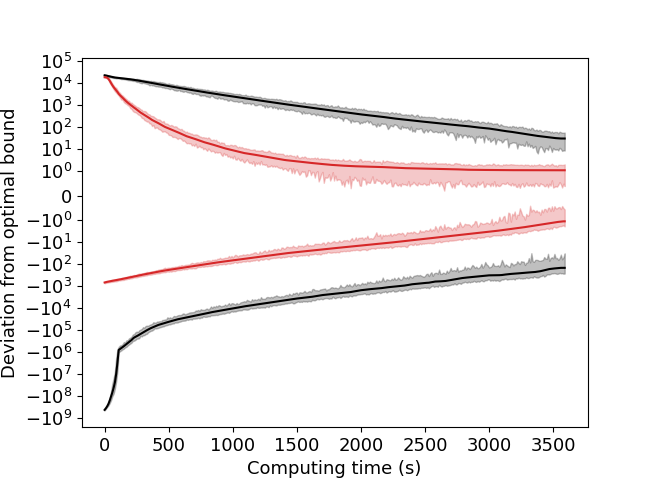}
     \caption{\textit{High maximum demand} dataset, 400 nodes}
     \label{forte_demande_400}
\end{minipage}
\end{figure}

\begin{figure}
\begin{minipage}{0.49\textwidth}
     \centering
     \includegraphics[width=\textwidth, trim={0 0 45 36},clip]{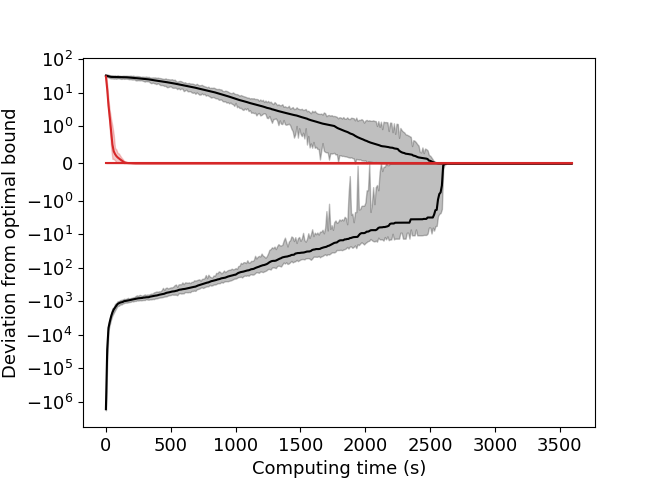}
     \caption{\textit{Size of commodities} dataset, capacity 100}
     \label{capacity_100}
\end{minipage}
\hfill
\begin{minipage}{0.49\textwidth}
     \centering
     \includegraphics[width=\textwidth, trim={0 0 45 36},clip]{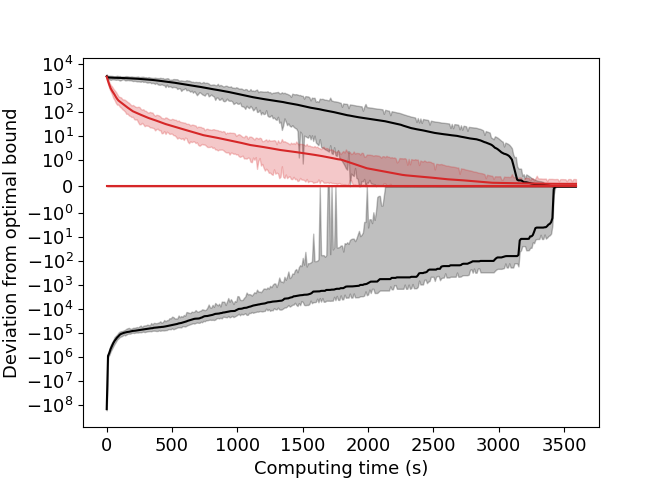}
     \caption{\textit{Size of commodities} dataset, capacity 10000}
     \label{capacity_10000}
\end{minipage}
\end{figure}

\textit{Solving the Fenchel sub-problem with a secondary normalization.} The new method of solving the Fenchel sub-problem presented in Section \ref{iterative_sub_problem} is used in the DW-Fenchel-iterative method. This method appears to be slightly slower than the DW-Fenchel method which uses a direct approach for the sub-problem. On the other hand, the direct approach sometimes fails to solve the sub-problem because of numerical instabilities which prevent the decomposition method from converging. For instances with 145 nodes of the \textit{Low maximum demand} dataset, this happens every 10 to 20 instances. The 10 instances presented in Figure \ref{faible_demande_145} did not suffer from instability in this set of experiments thus it does not appear in the figure. However, we were able to identify seeds for which the instability appears. Unfortunately, these seeds appear to be hardware-dependent thus researchers trying to reproduce the results will have to find their own seeds.

\textit{Impact of coupling the two master problems using directional normalization.} The DW-Fenchel and DW-Fenchel-iterative methods couple the Dantzig-Wolfe and Fenchel master problems using a directional normalization in their Fenchel sub-problem. The impact of this coupling can be studied by comparing these methods to the Fenchel method whose only difference is to use the natural normalization of the unsplittable flow problem in its sub-problem. As illustrated in Figure \ref{faible_demande_90}, one notes that all the methods are similar during the first iterations. However, the Fenchel method stalls rapidly. On the other hand, this is not the case with the methods using directional normalization which yield much better results. Our interpretation of this phenomenon is as follows. There are many equivalent optimal solutions of the linear relaxation of the unsplittable flow problem. When a cut is generated with the natural normalization, it only cuts a subset of these solutions, and the new solution to the Fenchel master problem is in a completely different place in the solution space. The method then fails to cut all the solutions because of this large number of symmetries. In contrast, the cuts generated using directional normalization focus on the optimal solution of the Dantzig-Wolfe master problem and try to prove its optimality. By focusing on a sub-part of the solution space, this method avoids the problem of symmetries which improves its convergence. This hypothesis is supported by the results of a preliminary study on a variation of the unsplittable flow problem where a path is favored. Indeed, because of the presence of a privileged path for each commodity, this variant does not have as many symmetries. In this context, there is a smaller difference between the methods based on the two normalizations.

\textit{Comparison between DW-interior-point and DW-Fenchel-iterative.} In the case of the \textit{low maximum demand} dataset where the number of commodities is high, the DW-Fenchel-iterative method behaves a lot better than the Dantzig-Wolfe methods. We assume that this is because the large number of commodities implies a greater degeneracy of the master problem which should less bother the DW-Fenchel methods. Indeed, this degeneracy seems to be the cause of the rather slow start of the Dantzig-Wolfe methods on these instances. In contrast, for the \textit{high maximum demand} dataset, the DW-interior-point and DW-Fenchel-iterative methods show more similar results. On these instances, the DW-Fenchel-iterative method shows a faster start of convergence, but slows down at the end of convergence, in particular for the lower bound.

\textit{Impact of capacity size.} The knapsack problem is known to have pseudo-polynomial resolution algorithms in the knapsack capacity such as the MIN\-KNAP algorithm that we use. In the case of unsplittable flows, this capacity of the knapsack corresponds to the capacity of the arcs. In Figure \ref{capacity_100} and \ref{capacity_10000}, we vary the capacities of the arcs. Note that the results for capacities of 1000 are given in Figure \ref{small_all_algo}. This variation in capacities impacts the computation time of the two methods making them slower. However, the DW-Fenchel-iterative method is much more impacted because the methods having a Fenchel sub-problem spend more time in their sub-problem than the Dantzig-Wolfe methods. This emphasizes the fact that having a fast optimization oracle for the polyhedron $Q_2$ is much more important for the methods based on a Fenchel sub-problem than those based on a Dantzig-Wolfe sub-problem.

\textit{General comments.} The new methods presented that couples the Dantzig-Wolfe and Fenchel decompositions shows very promising results. In particular, they seems to be far less affected by degeneracy than the Dantzig-Wolfe decomposition and possess better convergence than the Fenchel decomposition. On the other hand, they can end up converging slightly less rapidly than the Dantzig-Wolfe decomposition on instances where degeneracy is not an issue. The new methods are particularly effective when the optimization oracle ($ O $) can be implemented by a fast algorithm.

\section{Conclusions}

In this work we have revisited Dantzig-Wolfe and Fenchel decompositions for some hard combinatorial problems with block structures. We have provided geometrical and intuitive interpretations of several types of normalizations used in the literature to stabilize the sub-problems associated. This intuition has fueled the development of a novel methodology capable of coupling both decomposition approaches acting in tandem via a directional normalization. We have conducted a thorough computational campaign to demonstrate the effectiveness of the newly proposed approach for the unsplittable flow problem. We have observed that on problems suffering from high degrees of degeneracy, the new approach is superior to its competitors. Meanwhile, it is also competitive with the classical approaches on the less degenerate cases. We also proposed a new approach to solve the Fenchel subproblem with directional normalization by using an alternative normalization. We provide theoretical guarantees for the finiteness of this new approach for several classes of alternative normalizations and our experimental campaign revealed that it presents far less numerical instabilities.


A likely lead for future research will therefore be to investigate the performance of this new method in different contexts than the unsplittable flow problems. Moreover, one of the central points of this new method is the use of directional normalization in the Fenchel sub-problem. It would be interesting to use this normalization inside other decomposition methods.


\printbibliography

\newpage

\appendix

\section{Proof of the finite convergence for the iterative resolution of the Fenchel sub-problem}
\label{proof_subproblem}

In this appendix, we give proofs of the finite convergence for the method presented in Section \ref{iterative_sub_problem} for two types of secondary normalizations. First, when the secondary normalization guarantees that a facet of $ Q_2 $ will be generated by the linear program. Second, when the secondary normalization is $ \| \pi \| \leq 1 $ where $\|.\|$ is any norm.

\subsection{Secondary normalization generating facets}
\label{facet-normalization}

In this section, we are interested in the termination of the proposed method to solve the Fenchel sub-problem when the secondary normalization guarantees that a facet of $ Q_2 $ will be generated by the separation linear program.

\begin{theorem}
If the secondary normalization guarantees the generation of facets then the method generates a cut associated with a directional normalization in finite time.
\end{theorem}

\begin{proof}
We will show that in the worst case the method ends after the secondary separation has generated all the facets of $ Q_2 $. For this, we show that each facet of $ Q_2 $ is generated at most once. As the algorithm progresses, the point $ x '$ separated during secondary separations advances along the segment $ (\widecheck{x}, \widehat{x}) $ in the direction of $ \widehat{x} $. Once a facet of $ Q_2 $ is generated, the point $ x '$ is projected onto that facet along the segment $ (\widecheck{x}, \widehat{x}) $. The future points $ x '$ will therefore all satisfy the inequality associated with this facet which can thus no longer be generated by the secondary separation problem. Since a polyhedron has a finite number of facets, a facet intersecting the segment $ (\widecheck{x}, \widehat{x}) $ is generated in a finite number of steps. Once this happens, the method ends after a single call to the alternate separation problem. Indeed, the procedure places the point $ x '$ on the point of intersection between this facet and the segment $ (\widecheck{x}, \widehat{x}) $. On the next iteration, the secondary separation problem indicates that the point $ x '$ belongs to $ Q_2 $ and the method stops. \qed
\end{proof}

\subsection{Secondary normalization using any norm}
\label{norm-normalization}

We now present a proof of convergence when the secondary normalization is $ \| \pi \| \leq 1 $ where $\|.\|$ is any norm. This proof uses the Lemmas \ref{lemme_coupe_sous_gradient} and \ref{lemme_angle} presented in \citet{boyd1995convergence}. Previously, we recall the following properties and notations. In the case of a normalization $ \| \pi \| \leq 1 $, the dual problem of the separation problem is: find the point of $ Q_2 $ minimizing the distance $ \| x - \widecheck{x} \|^* $ where $ \|. \|^* $ is the dual norm of $ \|. \| $: $ \| \lambda \|^* = \max_{\| x \| \leq 1} \lambda x $. The solution point of the dual problem is denoted $ x^* $ and always satisfies at equality the cut generated by the separation problem. We now present the lemmas used in the proof.

\begin{lemma}[\citet{boyd1995convergence}]
\label{lemme_coupe_sous_gradient}
Let $\lambda^* x \leq \lambda^* x^*$ be the cut generated during the separation of a point $\widehat{x}$ from a polyhedron $Q_2$ with the normalization $\| \pi \| \leq 1$ where $x^*$ is the optimal solution of the dual of the separation problem. Then $-\lambda^*$ is a sub-gradient of $x \mapsto \| x - \widecheck{x} \|^*$ in $x^*$.
\end{lemma}

In the second lemma $\angle(\lambda, x)$ will denote the angle between the vectors $\lambda$ and $x$ (the one lower than $\pi$ radiant).

\begin{lemma}[\citet{boyd1995convergence}]
\label{lemme_angle}
For each norm, there exists an angle $ \theta_{min}> 0 $ such that at any point $ x $ and for any sub-gradient $ \lambda $ of this norm in $x$:
$$\angle(\lambda, x) \geq \frac{\pi}{2} - \theta_{min}$$
\end{lemma}

We now present the main theorem of this section.

\begin{theorem}
If the secondary normalization is $ \| \pi \| \leq 1 $ for any norm then the method generates a cut associated with a directional normalization in finite time.
\end{theorem}

\begin{proof}
\textit{Scheme of the proof:} First we will show that once a face intersecting the segment $ (\widecheck{x}, \widehat{x}) $ has been generated the method ends after a single call to the secondary separation problem. Secondly, we will show that if another face of $ Q_2 $ is generated, the point $ x '$ advances more than $\epsilon$ along the segment $ (\widecheck{x}, \widehat{x}) $ toward $ \widehat{x} $ where $\epsilon$ is a strictly non-negative distance independent of the iteration. Thus, this second case cannot happen more than $ \frac{\| \widehat{x} - \widecheck{x} \|_2}{\epsilon} $ times so the procedure ends in a finite number of steps.

\vspace{0.2cm}

\textbf{1)} Suppose that at one iteration, a face intersecting the segment $ (\widecheck{x}, \widehat{x}) $ is generated. After generating the face, the procedure places the point $ x '$ on the intersection point. At the next iteration, the secondary separation problem indicates that the point $ x '$ belongs to $ Q_2 $ and the method stops.

\vspace{0.2cm}

\textbf{2)} For the rest of this proof, we will denote by $ x^{(i)} $ the point $ x'$ separated during iteration $ i $ of the algorithm. We will show that if a face $ F $ of $ Q_2 $ generated does not intersect the segment $ (\widecheck{x}, \widehat{x}) $, then the point $ x '$ advances along the segment $ (\widecheck{x}, \widehat{x}) $ a strictly non-negative distance $d_{min} \sin(\theta_{min})$ independent of the face $ F $, \textit{i.e.} $ \| x^{(i + 1)} - x^{(i)} \|_2 \geq d_{min} \sin(\theta_{min}) $. To that end, we will use trigonometry on the triangle formed by the points associated with $x^{(i)}$, $x^*$, and $x^{(i+1)}$ where $x^*$ is the dual optimal solution of the secondary separation problem. This triangle will be denoted $\triangle x^{(i)} x^* x^{(i+1)}$. We will use $\theta_{min}$ as lower bound for the angle $\angle(x^{(i+1)} - x^{*}, x^{(i)} - x^{*})$ and $d_{min}$ as lower bound for the distance $\|x^{(i)} - x^{*} \|_2$.

\vspace{0.2cm}

\textbf{2.1)} Suppose that at iteration $i$, the secondary separation problem of $x^{(i)}$ returns a primal-dual solution pair $(\lambda^*, x^*)$ which thus correspond to the cut $\lambda^* x \leq \lambda^* x^*$. Recall that $x^*$ is a point of $Q_2$ satisfying the generated cut to equality. It is thus on the generated face $F$. From Lemma \ref{lemme_coupe_sous_gradient}, since the secondary normalization is $\| \pi \| \leq 1$, the vector $-\lambda^*$ is a sub-gradient of $\| x - x^{(i)} \|^*$ in $x^*$. Thus, according to Lemma \ref{lemme_angle}, we have $\angle(-\lambda^*, x^* - x^{(i)}) \leq \frac{\pi}{2} - \theta_{min}$ for a $\theta_{min} > 0$ depending on the used norm but neither on the face $F$ nor on the iteration. This is equivalent to $\frac{\pi}{2} - \angle(\lambda^*, x^{(i)} - x^*) \geq \theta_{min}$. After generating the cut $\lambda^* x \leq \lambda^* x^*$, the procedure projects the point $x^{(i)}$ on the hyperplane $\lambda^* x = \lambda^* x^*$ along the segment $(\widecheck{x}, \widehat{x})$. The result of this projection is the point $x^{(i+1)}$. Since both $x^{(i+1)}$ and $x^*$ are point of the hyperplane $\lambda^* x = \lambda^* x^*$, the vector $x^{(i+1)} - x^*$ is on the hyperplane $\lambda^* x = 0$. Let us consider the smallest angle between $x^{(i)} - x^*$ and a point of the hyperplane $\lambda^* x = 0$. Firstly, it is smaller than the angle $\angle(x^{(i+1)} - x^*, x^{(i)} - x^*)$. Secondly, this smallest angle can be expressed in terms of the normal $\lambda^*$ as $\frac{\pi}{2} - \angle(\lambda^*, x^{(i)} - x^*)$ which is shown above to be greater than $\theta_{min}$. Thus, we have shown that \textbf{the angle $\bm{\angle(x^{(i+1)} - x^{*}, x^{(i)} - x^{*})}$ is greater than $\bm{\theta_{min}}$}.

\vspace{0.2cm}

\textbf{2.2)} Let $d_{min}$ be the distance in norm $\|.\|_2$ between the segment $(\widecheck{x}, \widehat{x})$ and the union of the faces of $Q_2$ that do not intersect the segment $(\widecheck{x}, \widehat{x})$. Since the generated face $ F $ does not intersect the segment $(\widecheck{x}, \widehat{x})$, the distance between the segment $(\widecheck{x}, \widehat{x})$ and the face $F$ is greater than $d_{min}$. However, $x^{(i)}$ belongs to the segment $(\widecheck{x}, \widehat{x})$ and $x^*$ to the face $F$ so \textbf{the distance $\bm{\|x^{(i)} - x^{*} \|_2}$ is greater than $\bm{d_{min}}$}.

\vspace{0.2cm}

\textbf{2.3)} We will now show with simple trigonometry on the triangle $\triangle x^{(i)} x^* x^{(i+1)}$ that $\| x^{(i+1)} - x^{(i)} \|_2 \geq d_{min} \sin(\theta_{min})$. First, the distance $\| x^{(i+1)} - x^{(i)} \|_2$ is larger than the length of the altitude of the triangle $\triangle x^{(i)} x^* x^{(i+1)}$ associated with $x^{(i)}$. However, the length of this altitude is $\| x^{(i)} - x^* \|_2 \sin(\angle(x^{(i)} - x^*, x^{(i+1)} - x^*))$ which according to paragraph 2.1) and 2.2) is greater than $d_{min} \sin(\theta_{min})$. Thus finally, we have: $\bm{\| x^{(i+1)} - x^{(i)} \|_2 \geq d_{min} \sin(\theta_{min})}$.

\vspace{0.2cm}

The distance $d_{min} \sin(\theta_{min})$ is strictly non-negative and independent of the iteration of the algorithm which completes the proof.
\qed
\end{proof}

\section{Solution of the knapsack oracle}
\label{annex_knapsack}

All the decomposition methods presented in this work assume that there exists an efficient algorithm capable of optimizing a linear function on the polyhedron $ Q_2 $ which in our version of the unsplittable flow problem can be written as:
$$\{(f_a^k \in \{0, 1\})_{a \in A, k \in k} o_a \in \mathbb{R}^+ | \sum_{k \in K} f_a^k D_k \leq c_a + o_a\}.$$
Optimizing a linear function whose coefficients are $(\pi^k)_{k \in K}$ and $-\pi_o$ on this polyhedron can thus be done by following \milp:

\begin{subequations}
\begin{alignat*}{3}
(O_a) \quad &\max_{f_a^k, o_a} && \sum_{k \in K} \pi^k f_a^k - \pi_o o_a  \\
& \text{subject to} \quad && \sum_{k \in K} f_a^k D_k \leq c_a + o_a \\
&&& f_a^k \in \{0, 1\}, ~ o_a \in \mathbb{R}^+
\end{alignat*}
\end{subequations}

This problem can be addressed by solving two classic 0-1 knapsack problems using a case disjunction. This method was presented by \citet{buther2012reducing} but we recall it here. Consider the following disjunction: either the flow of the commodities respects the capacity of the arc $ a $, or the flow of the commodities exceeds the capacity of the arc $ a $. Finding the best solution in the first case amounts to solving the following problem:
\begin{subequations}
\begin{alignat*}{3}
&\max_{f_a^k} && \sum_{k \in K} \pi^k f_a^k \\
&\text{subject to} \quad && \sum_{k \in K} f_a^k D_k \leq c_a \\
&&& f_a^k \in \{0, 1\}
\end{alignat*}
\end{subequations}

Indeed, we suppose that the coefficient $ \pi_o $ is non-negative because otherwise the problem $ O_a $ would be unbounded. Since the flow of commodities respects the capacity of the arc, the variable $ o_a $ always takes the value zero and can be removed from the problem. Note that in this first case of the disjunction, the problem to be solved is a classic knapsack problem. In the second case of disjunction, finding the best solution amounts to solving the following problem:

\begin{subequations}
\begin{alignat*}{3}
&\max_{f_a^k, o_a} && \sum_{k \in K} \pi^k f_a^k - \pi_o o_a  \\
&\text{subject to} \quad && \sum_{k \in K} f_e^k D_k \leq c_a + o_a \\
&&& \sum_{k \in K} f_a^k d^k \geq c_a\\
&&& f_a^k \in \{0, 1\}, ~ o_a \in \mathbb{R}^+
\end{alignat*}
\end{subequations}

Since we assume that $ \pi_o $ is non-negative and that the flow of commodities does not respect the capacity of the arcs, the variable $ o_e $ is always equal to $ \sum_{k \in K} f_a^k D_k - c_a $. By performing the replacement in the objective function, by replacing the variables $ f_a^k $ by their complement $ \bar{f}_a^k = 1 - f_a^k $ and by multiplying the constraint $ \sum_{k \in K } f_a^k D_k \geq c_a $ by -1 we get the following reformulation:

\begin{subequations}
\begin{alignat*}{2}
&\max_{\bar{f}_a^k} && \sum_{k \in K} (\pi_o D_k - \pi^k) \bar{f}_a^k + C  \\
&\text{subject to} \quad && \sum_{k \in K} \bar{f}_a^k D_k \leq \sum_{k \in K} D_k - c_a \label{eq:inverse_knapsack_constraint}\\
&&& \bar{f}_a^k \in \{0, 1\}
\end{alignat*}
\end{subequations}

where the constant $ C $ is equal to $ \sum_{k \in K} (\pi^k - \pi_o d^k) - \pi_o c_a $. This reformulation shows that the problem to be solved in the second case of the disjunction is also a classic knapsack problem.

\textit{Note:} in the experimental study of Section \ref{sec:results}, in order to solve these two knapsack problems, we use the MINKNAP algorithm proposed by \citet{pisinger1997minimal}. Since MINKNAP only accepts integer-valued weights and profits, we multiply the dual vector $(\pi, \pi_0, \pi_o)$ by $10^{7}$ and truncate before invoking MINKNAP.

\section{Pre/post-processing for Fenchel subproblems}
\label{pre_post_traitement}

Solving directly a Fenchel subproblem is sometimes too computationally expensive to be integrated into a decomposition method. However, the resolution time of this subproblem can be greatly reduced with pre/post-processing steps. Indeed, \citet{boccia2008cut} showed that it is possible to solve the Fenchel separation problem by focusing on a sub-polyhedron of $Q_2$ of far lower dimension. However, the cut generated may not be valid for $Q_2$ and one must use a \textit{lifting} procedure to create a cut valid for $Q_2$. These concepts are explained in the following.

\textbf{Dimensionality reduction:} This technique typically applies when the variables $ x $ of the initial problem ($ P $) are binary. Instead of separating a point $ \widecheck{x} $ from $ Q_2 $, this point is separated from the sub-polyhedron $ Q_2^f(\widecheck{x}) $ induced by the variables taking a value different from their bounds in $ \widecheck{x} $ (in the binary case, induced by variables taking a fractional value in $ \widecheck{x} $). More precisely, $ Q_2^f (\widecheck{x}) = \{x \in Q_2 | x_i = \widecheck{x}_i \text { if } \widecheck{x}_i \in \{0, 1 \} \} $. Indeed, $ \widecheck{x} $ belongs to $ Q_2 $ if and only if it belongs to $ Q_2^f (\widecheck{x}) $ which implies that $ \widecheck{x} $ is separable from $ Q_2 $ if and only if it is separable from $ Q_2^f (\widecheck{x}) $. Replacing $ Q_2 $ with $ Q_2^f (\widecheck{x}) $ in the separation problem has two advantages. First, the problem only considers variables taking a fractional value in $ \widecheck{x} $ which reduces the number of dimensions of the problem and therefore speeds up its resolution. Forcing certain variables to take a binary value in the optimization oracle on $ Q_2 $ is also often easy to do. Second, if the separation problem ($ S $) is solved directly through constraint generation, most of the computational time is used to generate constraints associated with vertices of $ Q_2 $ not belonging to $ Q_2^f (\widecheck{x}) $ and which are therefore mostly unnecessary. Reducing the search space to $ Q_2^f (\widecheck{x}) $ allows the computation to concentrate on the part of the space containing the most important vertices.

\textbf{Lifting:} Upon applying the dimensionality reduction described before, the cut generated will be valid for $ Q_2^f (\widecheck{x}) $ but not necessarily for $ Q_2 $. A procedure able to create a valid cut for $ Q_2 $ from one valid for $ Q_2^f (\widecheck{x}) $ is the sequential lifting procedure \citep{wolsey1999integer} that we will now describe. In the following, we will name fixed variables, the variables eliminated from the separation problem when replacing the polyhedron $ Q_2 $ by the polyhedron $ Q_2^f (\widecheck{x}) $. For the fixed variables, the coefficients of the generated cut are all zero. At each step of the sequential lifting procedure, a new value for the coefficient of one of the fixed variables is computed so that the cut becomes valid for the polyhedron where this variable is no longer fixed. These new coefficients are optimal in the sense that if the initial cut was a facet of the polyhedron $ Q_2^f (\widecheck{x}) $ then the cut resulting from the sequential lifting is a facet of $ Q_2 $. We now present an iteration of the sequential lifting procedure. Suppose therefore that $ Q_2^f (\widecheck{x}) = \{x \in Q_2 | x_1 = 1 \} $; the procedure will then contain only one iteration because only the variable $ x_1 $ is fixed. The cut generated for $ Q_2^f (\widecheck{x}) $ is $ \pi x \leq \pi_0 $ where the coefficient $ \pi_1 $ associated with the variable $ x_1 $ is zero. The lifting procedure consists in creating a new cut $ \gamma x \leq \gamma_0 $ such that the two cuts are identical when $ x_1 = 1 $ and such that the new cut is valid for $ Q_2 $. Since $ \pi x \leq \pi_0 $ is already valid for $ Q_2^f (\widecheck{x}) $, the two previous conditions can be written:

$$\forall x \in \{x | x_1 = 1\}, ~ \gamma x - \gamma_0 = \pi x - \pi_0$$
$$\max_{x \in Q_2 ~|~ x_1 = 0} \gamma x \leq \gamma_0$$

Let $ e_i $ be the $ i^{\text{th}} $ vector of the canonical basis. Applying the first condition for $ x = e_1 $, we get $ \gamma_1 - \gamma_0 = \pi_1 - \pi_0 $. Moreover, by applying it in $ x = e_1 + e_i $ for all $ i> 1 $, we obtain $ \gamma_1 + \gamma_i - \gamma_0 = \pi_1 + \pi_i - \pi_0 $ which becomes $ \gamma_i = \pi_i $ by subtracting the previous equality. We therefore know:

\begin{subequations}
\begin{alignat*}{2}
& \forall i > 1, \gamma_i = \pi_i \\
& \gamma_1 = \gamma_0 + \pi_1 - \pi_0
\end{alignat*}
\end{subequations}

Thus, the value of $ \gamma_0 $  is the only element missing to know the value of all the $ \gamma_i $. By replacing $ \gamma_i $ by $ \pi_i $ for all $ i> 1 $ in the second condition we get $ \gamma_0 \geq \max_{x \in \{x \in Q_2 | x_1 = 0 \}} \pi x $. The smallest value of $ \gamma_0 $ satisfying the second condition is, therefore, $ \gamma_0 = \max_{x \in \{x \in Q_2 | x_1 = 0 \}} \pi x $ which can be computed with a call to the optimization oracle on $ Q_2 $ where the variable $ x_1 $ set to $ 0 $. It is therefore possible to compute each lifted coefficient using a single call to the oracle ($ O $).

\end{document}